\documentclass{tran-l}
\usepackage{enumerate}
\usepackage{amssymb}
\usepackage{xfrac}
\usepackage{tikz}
\usepackage{thm-restate}
\usepackage{hyperref}
\usepackage{cleveref}
\usepackage{tcolorbox}
\tcbset{
  mybox/.style={
    colback=gray!8,
    colframe=gray!60,
    boxrule=0.5pt
  }
}

\newtheorem{theorem}{Theorem}[section]
\newtheorem{lemma}[theorem]{Lemma}

\theoremstyle{definition}
\newtheorem{definition}[theorem]{Definition}
\newtheorem{example}[theorem]{Example}
\newtheorem{proposition}[theorem]{Proposition}
\newtheorem{corollary}[theorem]{Corollary}
\newtheorem{claim}[theorem]{Claim}

\theoremstyle{remark}
\newtheorem{remark}[theorem]{Remark}

\numberwithin{equation}{section}

\newcommand{\argmax}{\mathop{\mathrm{argmax}}\limits}

\def\R{\mathbb{R}}
\def\C{\mathcal{C}}

\let\oldtocsection=\tocsection

\let\oldtocsubsection=\tocsubsection

\let\oldtocsubsubsection=\tocsubsubsection

\renewcommand{\tocsection}[2]{\hspace{0em}\oldtocsection{#1}{#2}}
\renewcommand{\tocsubsection}[2]{\hspace{1em}\oldtocsubsection{#1}{#2}}
\renewcommand{\tocsubsubsection}[2]{\hspace{2em}\oldtocsubsubsection{#1}{#2}}

\begin{document}

\title[Ergodic optimization with linear constraints]{Ergodic optimization with linear constraints}

\author{Shengwen Guo}
\address{Department of Mathematics, St. Bonaventure University, St. Bonaventure, New York 14778}
%\curraddr{Department of Mathematics and Statistics, Case Western Reserve University, Cleveland, Ohio 43403}
\email{sguo@sbu.edu}

\author{Kevin McGoff}
\address{Department of Mathematics and Statistics, University of North Carolina at Charlotte, Charlotte, North Carolina 28223}
\email{kmcgoff1@charlotte.edu}

%    General info

%\date{January 1, 2001 and, in revised form, June 22, 2001.}



\begin{abstract}
Let $T : X \to X$ be a continuous map of a compact metrizable space, and let $\phi : X \to \mathbb{R}$ be a continuous function. The ergodic optimization problem is to  maximize the integral $\int \phi  \, d\mu$ as $\mu$ ranges over all $T$-invariant Borel probability measures on $X$. In this paper we consider a constrained version of the ergodic optimization problem. Given a `constraint set' $\C\subset C(X)$, let $M_\C(X,T)$ be the set of $T$-invariant Borel probability measures $\mu$ on $X$ such that $\int g \, d\mu = 0$ for all $g \in \C$. We investigate the problem of maximizing the integral $\int \phi \, d\mu$ over the constrained set $M_\C(X,T)$. We address basic properties of this optimization problem, beginning with nonemptiness of $M_\C(X,T)$ and existence of optimal solutions. Additionally, we establish the generic and prevalent uniqueness of optimal measures, we provide a realization result, and we give a characterization of the dual problem. This framework provides a common generalization of several previously considered optimization problems in dynamical systems and optimal transport.
\end{abstract}

\maketitle

% \section*{This is an unnumbered first-level section head}
% This is an example of an unnumbered first-level heading.

%% The correct journal style for \specialsection is all uppercase; a known bug
%% in amsart.cls prevents this, so input must be uppercase until it is fixed.
%\specialsection*{This is a Special Section Head}
% \specialsection*{THIS IS A SPECIAL SECTION HEAD}
% This is an example of a special section head%
%%%%%%%%%%%%%%%%%%%%%%%%%%%%%%%%%%%%%%%%%%%%%%%%%%%%%%%%%%%%%%%%%%%%%%%%
% \footnote{Here is an example of a footnote. Notice that this footnote
% text is running on so that it can stand as an example of how a footnote
% with separate paragraphs should be written.
% \par
% And here is the beginning of the second paragraph.}%
%%%%%%%%%%%%%%%%%%%%%%%%%%%%%%%%%%%%%%%%%%%%%%%%%%%%%%%%%%%%%%%%%%%%%%%%

%\tableofcontents

\section{Introduction} \label{Sect:Intro}

%Intro about ergodic optimization, including definition, motivation, references, and general direction of recent work.

Ergodic optimization has received substantial attention in recent decades; for an introduction with references, see the surveys \cite{jenkinson2006ergodic,jenkinson2019ergodic}. Throughout this work we assume that $X$ is a non-empty compact metrizable space and $T : X \to X$ is continuous, and we refer to the pair $(X,T)$ as a topological dynamical system. Let $M(X,T)$ denote the set of $T$-invariant Borel probability measures on $X$. The general problem of ergodic optimization can be described as follows. Let $(X,T)$ be a topological dynamical system, and let $\phi : X \to \mathbb{R}$ be a continuous function. Then the ergodic optimization problem is
\begin{equation} \label{Eqn:EO}
    \sup_{\mu \in M(X,T)} \int \phi \, d\mu.
\end{equation}
This problem is typically motivated by connections to maximal ergodic averages and the thermodynamic formalism. In particular, measures $\mu \in M(X,T)$ that achieve the supremum in (\ref{Eqn:EO}) appear as ground states or zero-temperature limits of equilibrium states. Much of the recent work on ergodic optimization has focused on establishing properties of the optimizing measures, particularly giving conditions guaranteeing that there is a unique optimizing measure that is supported on a periodic orbit (e.g., see \cite{contreras2016ground,huang2019ergodic,huang2026typical}).

Here we introduce the problem of ergodic optimization with linear constraints, which can be viewed as a localized or relativized version of the ergodic optimization problem. As the examples below illustrate, the framework of ergodic optimization with linear constraints serves as a common generalization of several optimization problems that have previously appeared in the literature, including (unconstrained) ergodic optimization and optimal transport. 

Let $T: X \to X$ and $\phi : X \to \mathbb{R}$ be as above, and let $C(X)$ be the space of continuous functions from $X$ to $\mathbb{R}$. For any subset $\C \subset C(X)$, we define
\begin{equation*}
    M_{\C}(X,T) = \left\{ \mu \in M(X,T) : \forall g \in \C, \, \int g \, d\mu = 0 \right\}.
\end{equation*}
We refer to $\C$ as the constraint set and $M_{\C}(X,T)$ as the feasible set. Then the problem of ergodic optimization with linear constraints is 
\begin{equation} \label{Eqn:EOWC}
\sup_{\mu \in M_{\C}(X,T)} \int \phi \, d\mu.
\end{equation}

Let us now present several classes of examples to illustrate some types of problems that can be described as ergodic optimization problems with linear constraints.

\begin{example}[Ergodic optimization] \label{Example1}
    If $\C = \varnothing$, then $M_{\C}(X,T) = M(X,T)$, and we recover the standard setting of unconstrained ergodic optimization.
\end{example}

\begin{example}[Ergodic optimization on a subsystem]
\label{Example2}
    Suppose $Y \subset X$ is non-empty, closed and satisfies $T(Y) \subset Y$, so that $(Y,T|_Y)$ is a subsystem of $(X,T)$. If we let $\C = \{ g \in C(X) : g|_Y \equiv 0 \}$, then any measure in $M_{\C}(X,T)$ is supported on $Y$, and ergodic optimization with constraint set $\C$ is equivalent to ergodic optimization on the subsystem $(Y,T|_Y)$. 
\end{example}

\begin{example}[Ergodic optimization over a rotation vector]
\label{Example3}
    Let $f_1,\dots,f_d$ be in $C(X)$, and let $\Phi : M(X,T) \to \R^d$ be given by 
    \begin{equation*}
        \Phi(\mu) = \left( \int f_1 \, d\mu, \dots, \int f_d \, d\mu \right). 
    \end{equation*}
    Here $\Phi(\mu)$ is called the rotation vector associated to $\mu$, and the image set $\Phi(M(X,T))$ is called a generalized rotation set. These sets have been studied in \cite{geller1999rotation, jenkinson2001rotation, kucherenko2014geometry, ziemian1995rotation}, and in particular, some authors have considered the thermodynamic formalism on sets of the form $\Phi^{-1}(v)$, where $v = (v_1,\dots,v_d) \in \R^d$ \cite{kucherenko2015localized}. If we choose $\C = \{ f_1 - v_1, \dots, f_d - v_d\}$, then we have $M_{\C}(X,T) = \Phi^{-1}(v)$, and hence ergodic optimization with these constraints corresponds to ergodic optimization over a particular rotation vector. Such constrained ergodic optimization has been considered in \cite{garibaldi2007functions}. 
\end{example}

\begin{example} [Relative ergodic optimization over a measure]
\label{Example4}
    %Relative EO over a measure. Note result of McGoff and Nobel.
    Suppose we have two topological dynamical systems $(X,T)$ and $(Y,S)$ together with a continuous map $\pi : X \to Y$ such that $\pi \circ T = S \circ \pi$. For a fixed measure $\nu \in M(Y,S)$, let 
    \begin{equation*}
        \C = \left\{ f \circ \pi - \int f \, d\nu : f \in C(Y) \right\}.
    \end{equation*}
    Then $M_{\C}(X,T)$ consists of all measures $\mu$ in  $M(X,T)$ that factor onto $\nu$ under $\pi$, i.e., $\pi_* \mu = \nu$. Thus the ergodic optimization problem with these constraints can be described as relative ergodic optimization over the measure $\nu$. Several papers have considered  measures of maximal relative entropy and the relative thermodynamic formalism over a measure \cite{allahbakhshi2019relative,ledrappier1977relativised,petersen2003measures,walters1986relative,yoo2017decomposition}. We also note that the problem of ergodic optimization with constraints of this type has appeared in the context of statistical inference for dynamical systems \cite{mcgoff2016variational}.
\end{example}

\begin{example}[Ergodic optimal transport] \label{Example5}
    Suppose we have two topological dynamical systems $(X_i,T_i)$ with measures $\mu_i \in M(X_i,T_i)$, for $i = 1,2$. Let $(X,T)$ be the direct product system, given by $X = X_1 \times X_2$ and $T(x_1,x_2) = (T_1(x_1),T_2(x_2))$, and let $\pi_i : X \to X_i$ be the coordinate projection map for $i = 1,2$. Now consider the constraint set 
    \begin{equation*}
        \C = \left\{ g \circ \pi_1 - \int g \, d\mu_1 : g \in C(X_1) \right\}  \cup \left\{ g \circ \pi_2 - \int g \, d\mu_2 : g \in C(X_2) \right\}.
    \end{equation*}
    Then $M_{\C}(X,T)$ is the set of all invariant measures for the product system that have marginals $\mu_1$ and $\mu_2$, i.e., $M_\C(X,T)$ is the set of all joinings of $\mu_1$ and $\mu_2$. Joinings were introduced by Furstenberg \cite{furstenberg1967disjointness} and have played an important role in ergodic theory (see the survey \cite{de2005introduction} and the book \cite{glasner2003ergodic}). Thus ergodic optimization with these constraints can be viewed as an ergodic optimal transport problem between $\mu_1$ and $\mu_2$. If $\phi : X_1 \times X_2 \to \R$ is the function to be maximized in the ergodic optimization problem, then $-\phi$ plays the role of the cost function from the point of view of optimal transport. We also note that the optimization problems that appear in the definitions of Ornstein's $\overline{d}$-metric \cite{ornstein1973application} and the more general $\overline{\rho}$-metric \cite{gray1975generalization} can be viewed as special cases of ergodic optimization with linear constraints of this type. Note that the `optimal transport' problem in \cite{lopes2012duality,lopes2015entropy} is over a different feasible set from the `ergodic optimal transport' problem here. %.  if $(X_1,T_1) = (X_2,T_2)$ and  with an appropriate choice of objective function $\phi \in C(X)$ in this setting, ergodic optimization with these constraints recovers Ornstein's $\overline{d}$-metric \cite{ornstein1973application}, as well as the more general $\overline{\rho}$-metric \cite{gray1975generalization}. 
\end{example}

\begin{remark}
    By taking the union of several constraint sets, one may consider constrained problems with combinations of the above constraint types. For example, by combining Examples \ref{Example3} and \ref{Example5}, one could consider ergodic optimal transport over a set of joinings with a fixed rotation vector. 
    %Note that one can let $\C$ be the union of the constraint sets among \Cref{Example1}-\ref{Example5} to explore ergodic optimization with both linear constraints. 
    %For instance, if we let
    %\[
    %\C = \{ f_1 - v_1, \dots, f_d - v_d\} \cup \left\{ f \circ \pi - \int f \, d\nu : f \in C(Y) \right\}
    %\]
    %be the union of the constraint sets in \Cref{Example3} and \Cref{Example4}, then (\ref{Eqn:EOWC}) turns to ergodic optimization over the set of all invariant measures with $\mu\circ\pi^{-1}=\nu$ and $\Phi(\mu)=v$.
\end{remark}

Our results address questions of existence, uniqueness, realization, and duality for ergodic optimization with linear constraints. We begin with a characterization of when the feasible set $M_{\C}(X,T)$ is non-empty and when solutions exist. Note that in the unconstrained case we have $M_{\C}(X,T) = M(X,T)$, which is known to be non-empty by the Krylov-Bogolioubov Theorem (see \cite[p.152]{walters2000introduction}). We prove the following theorem in Section \ref{Sect:Existence}.

\begin{restatable}[Existence]{theorem}{existence}
  \label{Thm:Existence}
  Let $(X,T)$ be a topological dynamical system and $\C \subset C(X)$. Let $H$ be the smallest closed $T$-invariant linear subspace of $C(X)$ containing $\C$.  Then the following are equivalent:
    \begin{enumerate}[(i)]
        \item $M_{\C}(X,T) = \varnothing$;
        \item there exists $g \in H$ such that $g(x) > 0$ for all $x \in X$.
    \end{enumerate}
    Furthermore, if $M_{\C}(X,T)$ is non-empty, then there exists $\mu^* \in M_{\C}(X,T)$ such that
    \begin{equation*}
        \int \phi \, d\mu^* = \sup_{\mu \in M_{\C}(X,T)} \int \phi \, d\mu.
    \end{equation*}
\end{restatable}

% \begin{theorem}[Existence]
%     Let $(X,T)$ be a topological dynamical system and $\C \subset C(X)$. Let $H$ be the smallest closed $T$-invariant linear subspace of $C(X)$ containing $\C$.  Then the following are equivalent:
%     \begin{enumerate}[(i)]
%         \item $M_{\C}(X,T) = \varnothing$;
%         \item there exists $g \in H$ such that $g(x) > 0$ for all $x \in X$.
%     \end{enumerate}
%     Furthermore, if $M_{\C}(X,T)$ is non-empty, then there exists $\mu^* \in M_{\C}(X,T)$ such that
%     \begin{equation*}
%         \int \phi \, d\mu^* = \sup_{\mu \in M_{\C}(X,T)} \int \phi \, d\mu.
%     \end{equation*}
% \end{theorem}

We refer to any measure $\mu \in M_{\C}(X,T)$ that achieves the supremum in (\ref{Eqn:EOWC}) as an optimal measure. Let $M_{\C}^*(X,T;\phi)$ denote the set of all optimal measures. 
When solutions exist, one may wish to know when they are unique, i.e., when there is precisely one optimal measure. In general there may be more than one optimal measure. For example, if $M_{\C}(X,T)$ contains more than one measure and $f$ is a constant function, then the solution set $M_{\C}^*(X,T;\phi)$ is equal to $M_{\C}(X,T)$, and we have non-uniqueness. (For more instances of non-uniqueness, see our realization results below.) On the other hand, the following result guarantees that there exists a unique solution for `typical' objective functions $\phi$, where `typical' is made precise in both topological and measure-theoretic senses. This theorem generalizes previous results addressing the unconstrained case \cite{jenkinson2006ergodic,morris2021prevalent}. To the best of our knowledge this type of result has not appeared in the settings of Examples  \ref{Example3}-\ref{Example5}.

To state the theorem, let $|A|$ denote the cardinality of any set $A$, and then let
\begin{equation*}
    U_\C = \bigl\{ \phi \in C(X) : |M_{\C}^*(X,T;\phi)|=1 \bigr\}.
\end{equation*}
A set is said to be residual if it is a countable intersection of sets with dense interior, giving a precise sense in which the set is topologically large. The notion of prevalence, introduced in \cite{hunt1992prevalence}, gives a precise sense in which a set is measure-theoretically large (see Section \ref{Sect:Prevalence} for details). Our proof of the following theorem appears in Section \ref{Sect:Uniqueness}. 

\begin{theorem}[Uniqueness] \label{Thm:UniquenessIntro}
    Let $(X,T)$ be a topological dynamical system and $\C \subset C(X)$. If $M_{\C}(X,T)$ is non-empty, then $U_\C \subset C(X)$ is both residual and prevalent.
\end{theorem}

By general principles, the solution set $M_{\C}^*(X,T;\phi)$ is a closed face of $M_{\C}(X,T)$ (see Section \ref{Sect:background}). One may then ask, which closed faces of $M_{\C}(X,T)$ can be realized as the solution set $M_{\C}^*(X,T;\phi)$ for some $\phi \in C(X)$? In the unconstrained case ($\C = \varnothing$), it is known that every closed face of $M(X,T)$ can be realized as the solution set $M_{\varnothing}^*(X,T;\phi)$ for some $\phi \in C(X)$ \cite[Theorem 3]{jenkinson2006every}. However, there are examples such that $M_{\C}(X,T) = M(X,T) \cap H$ for a finite-dimensional vector space $H$ and yet $M_{\C}(X,T)$ is \textit{not} a polytope (see \cite[Example 4.4]{phelps1969infinite}). In such examples there may be closed faces that are not `exposed' and therefore cannot be realized as $M_\C^*(X,T;\phi)$ for any $\phi \in C(X)$. These examples suggest that we should not expect the realization results for the unconstrained case to generalize to the constrained setting without additional assumptions. Here we present our realization results in terms of a `finite-type' assumption. Given $(X,T)$ and $\C$ as above, we say that the the triple $(X,T,\C)$ has finite-type if there exists a finite set $\C' \subset C(X)$ such that $M_{\C}(X,T) = M_{\C'}(X,T)$. We remark that Examples \ref{Example1}-\ref{Example5} above all have finite type (see Corollary \ref{Cor:finitetype} below). See Section \ref{Sect:background} for more discussion of this condition. We prove the following theorem in Section \ref{Sect:Realization}.

\begin{restatable}[Realization]{theorem}{realization}
\label{Thm:Realization}
    Let $(X,T)$ be a topological dynamical system and $\C \subset C(X)$. If $(X,T,\C)$ has finite type, then for every closed face $K$ of $M_{\C}(X,T)$, there exists a continuous function $\phi \in C(X)$ such that $K = M_{\C}^*(X,T;\phi)$.
\end{restatable}

% \begin{theorem}[Realization]

% \end{theorem}

%By Proposition \ref{Prop:finitetype} below, if $M_{\C}(X,T)$ is a face of $M(X,T)$, then $(X,T,\C)$ has finite type and Theorem \ref{Thm:Realization} gives that every closed face of $M_\C(X,T)$ can be realized as the solution set $M_\C^*(X,T; \phi)$ for some $\phi \in C(X)$. As a consequence, Corollary \ref{Cor:finitetype} below yields that Theorem \ref{Thm:Realization} applies in the settings of all of the Examples \ref{Example1}-\ref{Example5} above. 
%, as well as all problems with finite combinations of constraints of the types given in Examples \ref{Example1}-\ref{Example5}.

Finally, we provide a duality result for the problem of ergodic optimization with linear constraints. In the setting of compact metrizable spaces, our duality result is a generalization of Kantorovich duality from the theory of optimal transport. See \cite[Chapter 5]{villani2009optimal} for the standard result; see \cite[Theorem 2.3]{garibaldi2017ergodic} for a duality result in the setting of ergodic optimization for symbolic dynamical systems; and see \cite{zaev2015monge,lopes2012duality} for some duality results that apply to the setting of ergodic optimal transport. Our proof of the following theorem appears in Section \ref{Sect:Duality}.

\begin{theorem}[Duality]
\label{Thm:Duality}
 Let $(X,T)$ be a topological dynamical system and $\C \subset C(X)$. Let $H$ be the smallest linear subspace of $C(X)$ containing $\C$. Then for any objective function $\phi\in C(X)$, we have
\begin{equation*}
\sup_{\substack{\mu \in M_{\C}(X,T)}} \int \phi \,  d\mu = \inf_{\substack{f+g - g \circ T + c \geq \phi \\ f\in H, \, g\in C(X), \, c\in \R}} c.
\end{equation*}
\end{theorem}

%In Section \ref{}

\vspace{2mm}
\noindent
\textbf{Organization of the rest of the paper.} Section \ref{Sect:background} provides the necessary background and notation, as well as some preliminary results. Then in Sections \ref{Sect:Existence}-\ref{Sect:Duality} we prove our main results concerning existence, uniqueness, realization, and duality, respectively.

\section{Background and preliminary results} \label{Sect:background}

% Background on $C(X)$, $M(X,T)$, weak$^*$ topology, extreme points, ergodic decomposition, faces of $M(X,T)$, realization of continuous affine functionals by continuous functions. Define prevalence.

Let $X$ be non-empty, compact and metrizable, and let $\mathcal{B}(X)$ be the $\sigma$-algebra of all Borel subsets of $X$.
The space of real-valued continuous functions on $X$, denoted by $C(X)$, is a separable Banach space when equipped with the supremum norm, which we denote by $\|\cdot\|$. Let $C(X)^*$ denote the dual space of all continuous linear functionals on $C(X)$. Let $M(X)$ be the space of Borel probability measures on $X$, endowed with the weak$^*$ topology, and note that $M(X)$ is compact and metrizable. Now suppose $T : X \to X$ is continuous. A measure $\mu \in M(X)$ is said to be $T$-invariant if $\mu(T^{-1}(E)) = \mu(E)$ for all $E \in \mathcal{B}(X)$. Let $M(X,T)$ denote the set of all $T$-invariant Borel probability measures on $X$. A measure $\mu \in M(X,T)$ is said to be ergodic provided that $\mu(E)\in\{0,1\}$ for any set $E \in \mathcal{B}(X)$ such that $T^{-1}(E)=E$. Note that $M(X,T)$ is compact, convex and metrizable in the weak$^*$ topology. Furthermore, the extreme points of $M(X,T)$ are precisely the ergodic measures in $M(X,T)$. 

%Based on the properties of $M(X,T)$, we have the following properties of $M_\C(X,T)$.

Let $C(X)^{**}$ denote the continuous dual of $C(X)^*$, and 
define the natural (canonical) embedding $J:C(X)\to C(X)^{**}$, where $J(f) = J_f$ and $J_f$ is defined by  $J_f(\mu)=\mu(f)=\int f\, d\mu$. We denote the kernel of any linear functional $J$ by $\ker(J)$. Furthermore, let 
\begin{equation*}
    H_\C:=\bigcap_{f\in\C}\ker(J_f).
\end{equation*} The following lemma is an immediate consequence of the definitions.

\begin{lemma}
\label{Lemma:McapH}
Suppose $(X,T)$ is a topological dynamical system and $\C\subset C(X)$. Then $M_\C(X,T)=M(X,T)\cap H_\C$.
\end{lemma}
%\begin{proof}
%The result is obvious. First, any measure $\mu\in M(X,T)\cap H_\mathcal{C}$ equals $0$ on $\mathcal{C}$ (for any $f\in\mathcal{C}$, $\mu\in\ker(J_f)$, so $J_f(\mu)=\mu(f)=0$) and is invariant, so $\mu\in M_\mathcal{C}(X,T)$. We just need to show that $M_\mathcal{C}(X,T)\subseteq M(X,T)\cap H_\mathcal{C}$. For any $\mu\in M_\mathcal{C}(X,T)$, it must be in $M(X,T)$. It is also in $H_\mathcal{C}$ because for any $f\in\mathcal{C}$, we have $\mu(f)=0=J_f(\mu)$, that is, $\mu\in\ker(J_f)$ for any $f$, which implies $\mu\in\bigcap_{f\in\mathcal{C}}\ker(J_f)$.
%\end{proof}

%As a consequence of this lemma, we observe that $M_\C(X,T)$ is compact and convex. 
The following proposition records some basic properties of $M_\C(X,T)$. 

\begin{proposition}[Properties of $M_\C(X,T)$]
\label{Prop:M_C property}
    Suppose $(X,T)$ is a topological dynamical system and $\C\subset\C(X)$. Then %If $M_\C(X,T)$ is non-empty, then
\begin{enumerate}[(i)]
    \item $M_\C(X,T)$ is a compact subset of $M(X)$.
    \item $M_\C(X,T)$ is convex.
    \item Any ergodic measure $\mu$ in $M_\C(X,T)$ is an extreme point of $M_\C(X,T)$. 
    \item If $\mathcal{C}_1$ and $\mathcal{C}_2$ are subsets of $C(X)$, then
    \[
    M_{\mathcal{C}_1}(X,T)\cap M_{\mathcal{C}_2}(X,T)=M_{\C_1\cup\C_2}(X,T).
    \]
    % \item If $M_\mathcal{C}(X,T)$ satisfies the Facial Property and contains at least one ergodic measure, then any extreme point $\mu$ in $M_\mathcal{C}(X,T)$, then $\mu$ is ergodic.
    % \item If $\mu,\nu\in M_\C(X,T)$ are both ergodic and $\mu\neq\nu$ then they are mutually singular.
\end{enumerate} 
\end{proposition}

%The proof is very regular, please refer to \cite[Thm. 6.10]{walters2000introduction}.
\begin{proof}
By Lemma \ref{Lemma:McapH}, $M_\C(X,T)=M(X,T)\cap H_\C$, where $H_\C=\bigcap_{f\in\C}\ker(J_f)$. Observe that $\ker(J_f)$ is closed and convex for each $f \in \C$ (since $J_f$ is continuous and linear). Hence $M_\C(X,T)$ can be written as an intersection of closed convex sets with the compact  convex set $M(X,T)$, and therefore we obtain (i) and (ii). To prove (iii), note that if $\mu$ is ergodic in $M(X,T)$, then it is an extreme point in $M(X,T)$ \cite[Theorem 6.10 (iii)]{walters2000introduction}, and hence it is an extreme point in any convex subset of $M(X,T)$. Finally, to establish (iv) observe that for any $\mu \in M(X,T)$, we have that $\mu$ is in $M_{\C_1}(X,T) \cap M_{\C_2}(X,T)$ if and only if $\int g \, d\mu = 0$ for all $g \in \C_1 \cup \C_2$, and hence $\mu \in M_{\C_1}(X,T) \cap M_{\C_2}(X,T)$ if and only if $\mu \in M_{\C_1 \cup \C_2}(X,T)$. %For any $\mu\in M_{\C_1}(X,T)\cap M_{\C_2}(X,T)$, $\mu(f)=0$ whenever $f$ in $\C_1$ or in $\C_2$, so $\mu\in M_{\C_1\cup\C_2}(X,T)$. When $\mu\in M_{\C_1\cup\C_2}(X,T)$, $\mu$ must be in both $M_{\C_1}(X,T)$ and $M_{\C_2}(X,T)$ by letting $\C_1$ or $\C_2$ be empty. 
\end{proof}

The next proposition lists some basic properties of any optimal solution set $M_\C^*(X,T;\phi)$.

\begin{proposition} \label{Prop:BasicM*}
    Suppose that $(X,T)$ is a topological dynamical system and $\C \subset C(X)$. Let $\phi$ be in $C(X)$. %Then the set of optimal solutions to (\ref{Eqn:EOWC}), denoted by $M^*_\mathcal{C}(X,T;\phi)$, has the following properties.
    \begin{enumerate}[(i)]
        \item If $M_\C(X,T)$ is non-empty, then $M^*_\mathcal{C}(X,T;\phi)$ is non-empty, compact, and convex. 
        \item If $f,g\in C(X)$ and if there exists $c\in\R$ such that $f-g-c$ belongs to the closure of the set $\{h\circ T-h: h\in C(X)\}$, then $M_\C^*(X,T;f)=M_\C^*(X,T;g)$.
    \end{enumerate}
\end{proposition}

\begin{proof}\hfill
    \begin{enumerate}[(i)]
        \item Since $\phi\in C(X)$, the map from $M_\C(X,T)$ to $\R$ defined by $\mu\mapsto\int\phi \, d\mu$ is continuous. Thus, since $M_{\C}(X,T)$ is compact (by \Cref{Prop:M_C property}), by the Extreme value Theorem, if $M_\C(X,T)$ is non-empty, then $M_\C^*(X,T;\phi)$ is non-empty. Furthermore, by \Cref{Prop:M_C property}, $M_{\C}(X,T)$ is compact and convex, and then since the map $\mu \mapsto \int \phi \, d\mu$ is continuous and affine, one may easily check that  $M_\C^*(X,T;\phi)$ is compact and convex.
        \item For any measure $\mu \in M(X,T)$ and $h \in C(X)$, the $T$-invariance of $\mu$ gives that $\mu(h \circ T - h) = 0$, and then by continuity of $\mu$ as a linear functional on $C(X)$, we obtain that $\mu(k) = 0$ for any $k$ in the closure of $\{h \circ T - h : h \in C(X)\}$. Thus, if $f - g-c$ is in the closure of $\{h \circ T - h: h \in C(X)\}$, then $\mu(f) = \mu(g+c)$, and we see that $M_{\C}^*(X,T;f) = M_{\C}^*(X,T;g+c)$. Finally, we note that for any $c$ in $\R$ we have $\mu(g+c) = \mu(g) + c$ for all $\mu \in M(X,T)$, and thus $M_{\C}^*(X,T;g) = M_{\C}^*(X,T;g+c)$, which finishes the proof. %This is because $\int f \, d\mu=\int g \, d\mu$ as long as $f-g$ belongs to the closure of the set $\{h\circ T-h: h\in C(X)\}$.
    \end{enumerate}
\end{proof}

\subsection{Faces, simplices, and affine functionals}
%Convexity plays an important role in the analysis of $M_\C(X,T)$ and the problem of ergodic optimization with linear constraints. 
%Convexity is crucial for us to characterize the solutions to (\ref{Eqn:EOWC}) since both the feasible set $M_\C(X,T)$ and the solution set $M_\C^*(X,T;\phi)$ are convex. Moreover, some results of the convex analysis are helpful in Section \ref{Sect:Realization}.
Let us now introduce some basic concepts from convex analysis. See the books \cite{alfsen2012compact,phelps2002lectures} for more details. Let $K$ be a convex set. For any non-empty subset $G$ of $K$, the {\it convex hull} of $G$, denoted by $\mathrm{co}(G)$, is the smallest convex subset of $K$ containing $G$. A non-empty convex subset $F$ of $K$ is a {\it face} of $K$ provided that for any $v_1,v_2$ in $K$ and $\alpha\in(0,1)$, if $\alpha v_1+(1-\alpha)v_2\in F$, then $v_1$ and $v_2$ are in $F$. Note that a point $e$ of $K$ is an extreme point of $K$ if and only if $\{e\}$ is a face of $K$. We let $\mathrm{ext}(K)$ denote the set of extreme points in $K$. By Choquet's Theorem, if $K$ is compact and convex, then every point of $K$ can be represented as a convex combination of extreme points (in the sense of a probability measure on the extreme points); see \cite[p.14]{phelps2002lectures}. Furthermore, a compact convex set $K$ is a {\it simplex} if and only if each point of $K$ is uniquely represented as a convex combination of the points in $\mathrm{ext}(K)$; see \cite[Chapter 10]{phelps2002lectures} for precise statements and details. Finite-dimensional examples of simplices include triangles and tetrahedrons. More importantly for our purposes, if $(X,T)$ is topological dynamical system, then $M(X,T)$ is known to be a compact, metrizable simplex (see \cite[Theorem 6.10]{walters2000introduction}). 
Finally, let us recall that a functional $l: K\to\R$ is said to be {\it affine} if for all $v_1,v_2\in K$ and $\alpha\in\R$, we have
\[
l(\alpha v_1+(1-\alpha)v_2)=\alpha l(v_1)+(1-\alpha)l(v_2).
\]

% \begin{figure}[ht!]
% \begin{tikzpicture}[scale=1.0]

% % ================= Triangle =================
% \begin{scope}
% \coordinate (v1) at (0,0);
% \coordinate (v2) at (3,0);
% \coordinate (v3) at (1.4,2.3);

% \fill[gray!15] (v1) -- (v2) -- (v3) -- cycle;
% \draw[very thick] (v1) -- (v2) -- (v3) -- cycle;

% \fill (v1) circle (2.5pt) node[below left] {$v_1$};
% \fill (v2) circle (2.5pt) node[below right] {$v_2$};
% \fill (v3) circle (2.5pt) node[above] {$v_3$};

% \node at (1.5,-0.7) {Triangle simplex};
% \end{scope}

% % ================= Tetrahedron =================
% \begin{scope}[xshift=6cm]
% \coordinate (e1) at (0,0);
% \coordinate (e2) at (2.7,0);
% \coordinate (e3) at (3.5,1.5);
% \coordinate (e4) at (1.5,3);

% % Faces
% \fill[gray!12] (e1) -- (e2) -- (e4) -- cycle;
% \fill[gray!22] (e1) -- (e3) -- (e4) -- cycle;
% \fill[gray!18] (e2) -- (e3) -- (e4) -- cycle;

% % Edges
% \draw[very thick] (e1) -- (e2) -- (e4) -- cycle;
% \draw[very thick] (e1) -- (e4);
% \draw[very thick] (e2) -- (e4);
% \draw[very thick] (e3) -- (e4);
% \draw[very thick] (e2) -- (e3);

% % Hidden/base edges
% \draw[very thick, dashed] (e1) -- (e3);

% % Vertices
% \fill (e1) circle (2.5pt) node[below left] {$e_1$};
% \fill (e2) circle (2.5pt) node[below right] {$e_2$};
% \fill (e3) circle (2.5pt) node[right] {$e_3$};
% \fill (e4) circle (2.5pt) node[above] {$e_4$};

% \node at (1.5,-0.7) {Tetrahedron simplex};
% \end{scope}

% \end{tikzpicture}
% \caption{Examples of simplices.}
% \label{Fig:simplex}
% \end{figure}

\subsection{Facial property}

%Here is a very significant property of $M(X,T)$. $M(X,T)$ is a simplex and $\mathrm{ext}(M(X,T))$ is precisely the set of ergodic probability measures of $M(X,T)$. For detailed introduction and proof, see \cite[Chapter 12]{phelps2002lectures} and \cite[Thm. 6.10]{walters2000introduction}.

The ergodicity of the extreme points of $M(X,T)$ is a fundamental fact in ergodic theory. %, and it is closely related to the characterization of closed faces of $M(X,T)$ and the uniqueness of the relative entropy \cite{petersen2003measures} and relative equilibrium states \cite{allahbakhshi2019relative, yoo2017decomposition}. 
However, in contrast to the situation in the unconstrained case of $M(X,T)$, the extreme points of $M_\C(X,T)$ may not be ergodic.
%and therefore are difficult to characterize. 
Indeed, the following example from \cite{kucherenko2015localized} %The following example
has the property that $M_\C(X,T)$ contains no ergodic measures whatsoever.
%there might be no ergodic measure in $M_\C(X,T)$ for some dynamical systems $(X,T)$ with some selections of constraint sets $\C$.

\begin{example}[\cite{kucherenko2015localized}, Example 3]
Let $a,b,c,d\in\mathbb R$ with $a<b<c<d$. Let $X=[a,b]\cup[c,d]$ and $T: X\to X$ continuous. In addition, assume that $T([a,b])\subset[a,b]$, $T([c,d])\subset[c,d]$, $T(a)=a$, and $T(d)=d$. For $w\in(b,c)$, let $\mathcal{C}=\{\mathrm{Id}_X-w\}$ where $\mathrm{Id}_X$ is the identity map on $X$. Then the set $M_\C(X,T)$ does not contain any ergodic measures.
\end{example}

%To ensure that $\mathrm{ext}(M_\C(X,T))$ is precisely the ergodic measure of $M_\C(X,T)$, we proposed the {\it Facial property} (or {\it property (FP)}), which is based on the concept of ergodic decomposition. Since $M(X,T)$ is a simplex and the ergodicity of $\mathrm{ext}(M(X,T))$, by the Choquet representation theorem (see \cite[p.60]{phelps2002lectures}), each element of $M(X,T)$ can be uniquely represented by the ergodic measures of $M(X,T)$.
Next we recall the ergodic decomposition of invariant measures.

\begin{definition}[Ergodic decomposition]
Let $M^e(X,T)$ be the set of ergodic measures of $M(X,T)$. Then for each $\mu\in M(X,T)$ there is a unique probability measure $P_{\mu}$ on the Borel subsets of $M(X,T)$ such that $P_{\mu}(M^e(X,T))=1$ and 
\[
\mu=\int_{M^e(X,T)}m \, dP_{\mu}(m).
\]
\end{definition}

Now we define a special property of the triple $(X,T,\C)$. 

\begin{definition}[Facial Property] \label{Defn:Facial}
We say that $(X,T,\C)$ satisfies the {\it facial property} (or {\it FP}) if for any $\mu\in M_\C(X,T)$, its ergodic decomposition measure $P_\mu$ satisfies $P_{\mu}(M_\C(X,T)) = 1$. 
%if its ergodic decomposition is
%\begin{equation}
%\mu=\int_{M^e(X,T)}\nu\, d\tau(\nu),  \end{equation}
%then $\nu\in M_\C(X,T)$ for $\tau$-almost every $\nu$. Where $\tau$ is a probability measure defined on the Borel subsets of $M(X,T)$ and supported on $M^e(X,T)$.
\end{definition}

\begin{remark}
Suppose $M_\C(X,T)$ is non-empty. Then $(X,T,\C)$ has the facial property if and only if $M_\C(X,T)$ is a face of $M(X,T)$. Consequently, if $(X,T,\C)$ has the facial property, then the extreme points of $M_\C(X,T)$ are exactly the ergodic measures in $M_\C(X,T)$, and $M_\C(X,T)$ contains at least one ergodic measure.  %That is, $M_\C(X,T)$ is a closed face of $M(X,T)$ when the Facial Property holds.
\end{remark}

\begin{example} The following examples all have the facial property. %Examples that satisfy the Facial Property.

\begin{enumerate}[(i)]
    \item Unconstrained ergodic optimization (Example \ref{Example1}); %  ($\mathcal{C}=\{0\}$). 
    \item Ergodic optimization on a subsystem (Example \ref{Example2}); % ($\mathcal{C}=\{g\in C(X): g|_Y\equiv 0\}$).
    \item Relative ergodic optimization over an ergodic measure $\nu\in M(Y,S)$ (Example \ref{Example4}); % satisfies the facial property.
    \item Ergodic optimal transport between two ergodic measures $\mu_1\in M(X_1,T_1)$ and $\mu_2\in M(X_2,T_2)$ (Example \ref{Example5}). % satisfies the facial property. 
\end{enumerate}   
\end{example}

% Suppose there are two constraint sets $\C_1$ and $\C_2$, define
% \[
% \C_1+\C_2=\{h_1+h_2: h_1\in \mathcal{C}_1, h_2\in \mathcal{C}_2\},
% \]
% then we have the following results:

% \begin{enumerate}[(i)]
%     \item If $\mathcal{C}_1$ and $\mathcal{C}_2$ are subsets of $C(X)$ satisfying the facial property, then $\mathcal{C}_1+\mathcal{C}_2$ is a subspace of $C(X)$ that satisfies the facial property.
%     \item Let $C_T(X)=\{f\in C(X): f\circ T^{-1}=f\}$. If $\mathcal{C}_1$ and $\mathcal{C}_2$ are subspaces of $C_T(X)$, then $\mathcal{C}_1+\mathcal{C}_2\subset C_T(X)$.
%     \item If $\mathcal{C}_1$ and $\mathcal{C}_2$ are subspaces of $C(X)$, then
%     \[
%     M_{\mathcal{C}_1}(X,T)\cap M_{\mathcal{C}_2}(X,T)=M_{\mathcal{C}_1+\mathcal{C}_2}(X,T)=M_{\C_1\cup\C_2}(X,T).
%     \]
% \end{enumerate}

Recall that $M_\C^*(X,T;\phi)$ is the set of all optimal measures of (\ref{Eqn:EOWC}) with the given objective function $\phi\in C(X)$. When the facial property holds for $(X,T,\C)$, the set of optimal measures $M_\C^*(X,T;\phi)$ possesses some additional properties, which we state in the following proposition.

\begin{proposition}
\label{Prop:Property M_C^*}
Suppose that $(X,T)$ is a topological dynamical system and $\C \subset C(X)$, and let $\phi\in C(X)$. If the feasible set $M_\C(X,T)$ is non-empty and $(X,T,\C)$ has the facial property, then the extreme points of $M^*_\mathcal{C}(X,T;\phi)$ are precisely the ergodic measures in $M^*_\mathcal{C}(X,T;\phi)$, and $M^*_\mathcal{C}(X,T;\phi)$ contains an ergodic measure. %the set of optimal solutions to (\ref{Eqn:EOWC}), denoted by $M^*_\mathcal{C}(X,T;\phi)$, has the following properties.
%\begin{enumerate}[(i)]
%    \item the extreme points of $M^*_\mathcal{C}(X,T;\phi)$ are precisely the ergodic measures in $M^*_\mathcal{C}(X,T;\phi)$;
%    \item $M^*_\mathcal{C}(X,T;\phi)$ contains an ergodic measure.
%\end{enumerate}     
\end{proposition}
\begin{proof}%\hfill
Since any ergodic measure in $M(X,T)$ is an extreme point of $M(X,T)$, we have that any ergodic measure in $M_{\C}^*(X,T;\phi)$ is an extreme point of $M_{\C}^*(X,T;\phi)$. Now suppose that $M_\C(X,T)$ is non-empty and $(X,T,\C)$ has the facial property. Then $M_{\C}(X,T)$ is a face of $M(X,T)$. Furthermore, as $M_{\C}^*(X,T; \phi)$ is defined as the set of optimal measures for an affine functional on the convex set $M_{\C}(X,T)$, we get that $M_{\C}^*(X,T; \phi)$ is a face of $M_{\C}(X,T)$. Hence $M_{\C}^*(X,T;\phi)$ is a face of $M(X,T)$, and the extreme points of $M_{\C}^*(X,T;\phi)$ are extreme points of $M(X,T)$. Thus the extreme points of $M_{\C}^*(X,T;\phi)$ are ergodic. Finally, as $M_{\C}^*(X,T;\phi)$ is non-empty, compact and convex, it must contain an extreme point by the Krein-Milman Theorem, and therefore $M_{\C}^*(X,T;\phi)$ contains an ergodic measure.
%\begin{enumerate}[(i)]
%    \item This is because the extreme point of $M_\C^*(X,T;\phi)$ is also extreme in $M_\C(X,T)$ and the facial property.
 %   \item For $\mu\in M^*_\C(X,T;\phi)$, let $\mu=\int_{ M^e_\mathcal{C}(X,T)}m \, d\tau(m)$ be the ergodic decomposition of $\mu$, where $ M^e_\mathcal{C}(X,T)$ is the set of ergodic measures of $ M_\C(X,T)$. Then
  %  \[
   % \int\phi \, d\mu=\int\phi \, d\left(\int_{ M^e_\mathcal{C}(X,T)}m \, d\tau(m)\right)=\int_{ M^e_\mathcal{C}(X,T)}\left(\int\phi \, dm\right) \, d\tau(m).
    %\]
    %Since $\int\phi \, dm\leq\int\phi \, d\mu$, so $m\in M^*_\C(X,T;\phi)$ for $\tau$-almost all $m$.
%\end{enumerate}    
\end{proof}

\subsection{Representation of \texorpdfstring{$M_\C(X,T)$}{M_C(X,T)} and finite type constraints}

Next, we provide a realization result describing which subsets of $M(X,T)$ can appear as $M_\C(X,T)$ for some choice of constraint set $\C$. The following representation result for continuous affine functionals on $M(X,T)$ is helpful.

\begin{proposition}[\cite{jenkinson2006every}, Proposition 1]
\label{Prop:JenRepresentation}
    Suppose $l:M(X,T)\to\R$ is weak$^*$ continuous and affine. Then there exists $g\in C(X)$ such that
    \[
    l(\mu)=\int g \, d\mu\quad\text{for all}\ \mu\in M(X,T).
    \]
\end{proposition}

\begin{proposition}
\label{Prop:representation}
    For any compact and convex subset $K$ of $M(X,T)$, there is a subset $\C\subset C(X)$ such that $M_\C(X,T)=K$.    
\end{proposition}

\begin{proof}
Let $K$ be a compact convex subset of $M(X,T)$, and  
define 
\[
\C= \bigl\{f\in C(X): \mu(f)=0\ \text{for all}\ \mu\in K \bigr\}. %=\bigcap_{\mu\in K}\ker(\mu).
\]
Let us show that $K=M_\C(X,T)$. First, notice that $K\subset M_\C(X,T)$, since for any $\mu\in K$ and $f\in\C$, we have $\mu(f)=0$ by definition of $\C$, and therefore $\mu \in M_\C(X,T)$.

In order to establish the reverse inclusion, let $\mu_0\in M(X,T)\setminus K$. %In Section \ref{Sect:background}, $M(X,T)$ is a locally convex topological vector space with weak$^*$ topology. For any $\mu_0\in M_\mathcal{C}(X,T)\setminus K$, 
As $K$ and $\{\mu_0\}$ are compact, convex and disjoint, a corollary of the Hahn-Banach Theorem (\cite[Theorem 3.12]{aliprantis2006positive}) gives us a continuous affine functional 
%, two non-empty and disjoint subsets $K$ and $\{\mu_0\}$ can be separated by a continuous linear functional 
$\psi:M(X,T)\to\mathbb R$ such that $\psi(\mu_0)>0$ and $\psi|_K\leq 0$. Define $\alpha:M(X,T)\to\mathbb R$ by
\[
\alpha(\mu)=\max(\psi(\mu),0),
\]
and define $\beta:M(X,T)\to\mathbb R$ by
\[\quad \beta(\mu)=\begin{cases}
0,&\text{if}\ \mu\in K\\
\max\limits_{\nu\in M(X,T)}|\psi(\nu)|,& \text{if}\ \mu\in M(X,T)\setminus K
\end{cases}.
\]
Note that $\alpha$ is continuous and convex, and $\beta$ is lower semicontinuous and concave, and $\alpha\leq \beta$. Then by Edwards' Theorem (\cite[Theorem II.3.10]{alfsen1964geometry}), there exists a continuous affine functional $\ell$ on $M(X,T)$ such that $\alpha\leq\ell\leq\beta$. Hence we have
\[
\ell|_K\equiv 0, \quad \text{and}\quad \ell(\mu_0)>0.
\]
Finally, by \Cref{Prop:JenRepresentation}, there exists $f_0\in C(X)$ such that
\[
\ell(\mu)=\int f_0 \, d\mu\quad\text{for all}\ \mu\in M(X,T).
\]
Note that $f_0$ satisfies
\[
\mu(f_0)=\ell(\mu)=0,\ \text{for all}\ \mu\in K, \]
and hence $f_0 \in \C$. Additionally, we have
\[
\mu_0(f_0)=\ell(\mu_0)>0,
\]
and therefore $\mu_0$ is not in $M_\C(X,T)$. Hence $M_\C(X,T) \subset K$, and we conclude that $K = M_\C(X,T)$.
%That contradicts the fact that $f_0\in\mathcal{C}:=\cap_{\mu\in K}\ker(\mu)$ and $\mu_0\in M_\mathcal{C}(X,T)$. Therefore, $M_\mathcal{C}(X,T)\subseteq K$.
\end{proof}

\begin{definition}[Finite-type constraints]
Suppose $(X,T)$ is a topological dynamical system and $\C\subset C(X)$ is a constraint set. We say that the triple $(X,T,\C)$ has \textit{finite type} if there exists a finite constraint set $\C' \subset C(X)$ such that $M_{\C'}(X,T)=M_\C(X,T)$.
\end{definition}
When $(X,T,\C)$ has finite-type, the feasible set $M_\C(X,T)$ can be written as the intersection of the simplex $M(X,T)$ and finitely many hyperplanes of `codimension $1$' (by applying Lemma \ref{Lemma:McapH} with $\C'$). We remark that compact convex sets of this type have previously appeared in the literature \cite{phelps1969infinite}, where they are called {\it $\beta$-polytopes}.  %(see \cite{phelps1969infinite}). 
%We have the following properties of finite-type constraints.

\begin{proposition}
\label{Prop:finitetype}
    If $(X,T,\C)$ has the facial property, then there exists $\phi \in C(X)$ such that $M_{\C}(X,T) = M_{\{\phi\}}(X,T)$, and hence $(X,T,\C)$ has finite-type.
\end{proposition}
\begin{proof}
Since $M_\C(X,T)$ is closed and $(X,T,\C)$ has the facial property, $M_{\C}(X,T)$ is a closed face of $M(X,T)$. By \cite[Corollary 3.13]{fonf2001infinite}, there exists a weak$^*$ continuous affine functional $l:M(X,T)\to\mathbb R$ such that
\[
l|_{M_\C(X,T)}\equiv 0\quad\text{and}\quad l|_{M(X,T)\setminus M_\C(X,T)}<0.
\]
By \Cref{Prop:JenRepresentation}, there exists $\phi\in C(X)$ such that
\[
l(\mu)=\int \phi \, d\mu\quad\text{for all}\ \mu\in M(X,T).
\]
Letting $\mathcal{C}'=\{\phi\}$, we see that $M_\C(X,T)=M_\mathcal{C'}(X,T)$.
\end{proof}

\begin{corollary}
\label{Cor:finitetype}
    Each of the triples $(X,T,\C)$ in Examples \ref{Example1}-\ref{Example5} has finite type.
\end{corollary}
\begin{proof}
    Let $(X,T,\C)$ be as in one of Examples \ref{Example1}, \ref{Example2}, \ref{Example4}, and \ref{Example5}. Observe that $(X,T,\C)$ has the facial property, and hence it has finite type by Proposition \ref{Prop:finitetype}. 
Also, it is immediate from the definitions that any constrained system $(X,T,\C)$ as in Example \ref{Example3} has finite type.
\end{proof}

\section{Existence} \label{Sect:Existence}
When $(X,T)$ is a topological dynamical system, the Krylov-Bogolioubov Theorem guarantees that $M(X,T)$ is non-empty. %, and then the Extreme Value Theorem gives that a solution to the standard ergodic optimization problem (\ref{Eqn:EO}) always exists. 
However, $M_\C(X,T)$ may be empty for some systems $(X,T)$ and constraint sets $\mathcal{C}$. For trivial examples, consider $\mathcal{C}=C(X)$ or $\mathcal{C} = \{\mathbf{1}\}$, where $\mathbf{1}$ denotes the constant function equal to one. The following example shows that $M_\C(X,T)$ may be empty due to interactions between the dynamics and the constraint set.
\begin{example}[Circle Rotations]
Let $(S^1,R_{\alpha})$ be an irrational rotation of the circle. Viewing the circle $S^1$ as the quotient space $[0,1]/\{0 \sim 1\}$, we let $f \in C(X)$ be such that $f \geq 0$ and $f(x) >0$ for some $x \in S^1$, e.g., we could choose
\begin{equation*}
f(x) = \max\bigl( 0, -(1-x)(1/2-x)\bigr).
\end{equation*}
Letting $\C = \{f\}$, we obtain that $M_\C(X,T)$ is empty, since the unique element $\mu$ of $M(X,T)$ is the Lebesgue measure on $S^1$, and we have $\mu(f) >0$. 
%Let $S^1=[0,1]/\sim$ be the unit circle, where $\sim$ indicates that $0$ and $1$ are identified, and $\mod 1$ makes $S^1$ an abelian group. The natural distance on $[0,1]$ induces a distance on $S^1$: 
%\[
%d(x,y)=\min(|x-y|,1-|x-y|)
%\]
%Lebesgue measure on $[0,1]$ gives a natural measure $\lambda$ on $S^1$.
%For $\alpha\in\mathbb R$ irrational, let $R_\alpha$ be the rotation of $S^1$ by angle $2\pi\alpha$, i.e.
%\[
%R_\alpha x=x+\alpha\mod 1
%\]
%$(S^1,\mathcal{B}([0,1]),\lambda,R_\alpha)$ is a measure-preserving system. If $f=-(1-x)(1-\alpha-x)\mathbf{1}_{[1-\alpha,1]}\in\mathcal{C}$, then for any element in $M_\mathcal{C}(X,T)$, whose support must avoid $[1-\alpha,1]$. However, such measures cannot be measure-preserving since for any interval $I=[a,b]$ with positive measure and $b-a<\alpha$, there exists an $n\in\mathbb N$ s.t.  $R_\alpha^n I=[R_\alpha^n a, R_\alpha^n b]\subset[1-\alpha,1]$ (this is because the orbit $\{R^n_\alpha\}_n$ is dense in $S^1$), thus $R^nI=0$. So $M_\mathcal{C}(X,T)$ is empty in this setting.
\end{example}

%Thus, the solution to (\ref{Eqn:EOWC}) does not always exist.
%Whenever $M_\C(X,T)$ is empty, no solutions of (\ref{Eqn:EOWC}) can exist. 
%To establish the existence result for (\ref{Eqn:EOWC}), we should ensure the non-emptieness of $M_\C(X,T)$ first.
The following lemma gives a sufficient condition for nonemptieness of $M_\C(X,T)$, and it is helpful for the proof of Theorem \ref{Thm:Existence} below. The proof of this lemma involves a modification of the standard Krylov-Bogolioubov argument. A subset $H \subset C(X)$ is said to be $T$-invariant if $f \circ T \in H$ for all $f \in H$. 

\begin{lemma}
\label{Lem:Existence}
Let $(X,T)$ be a topological dynamical system and $\C \subset C(X)$. If $\C$ is $T$-invariant and there exists a Borel probability measure $\mu \in M(X)$ such that $\mu(f)=0$ for all $f\in \mathcal{C}$, %and for any $f\in\mathcal{C}$, we have $f\circ T\in\mathcal{C}$, 
then $M_\C(X,T)$ is non-empty.
\end{lemma}
\begin{proof}
As $X$ is a compact metrizable space, the Banach space $(C(X),\|\cdot\|)$ is separable. Let $\mu$ be a Borel probability measure on $C(X)$ such that $\mu(f) = 0$ for all $f \in \C$. For each $n \in \mathbb{N}$, define the linear functional $\psi_n : C(X) \to \mathbb{R}$ by setting
\[
\psi_n(f)
=\int_X\left[\frac{1}{n}\sum_{i=0}^{n-1}f\circ T^i\right] \, d\mu,\ \forall f\in C(X).
\]
It is clear that $|\psi_n(f)|\leq\|f\|$ for all $f\in C(X)$ and $n\in\mathbb N$, and hence each $\psi_n$ is bounded. Thus $\{\psi_n\}_{n=1}^{\infty}$ is a bounded sequence in $C(X)^*$. As $C(X)$ is a separable Banach space with the norm $\|\cdot\|$, by Helly's theorem there exists a subsequence $\{\psi_{n_k}\}_{k=1}^{\infty}$ that converges with respect to the weak$^*$ topology to a bounded linear functional $\psi\in C(X)^*$, i.e., for all $f \in C(X)$,
\[
\lim_{k\to\infty}\psi_{n_k}(f)=\psi(f).
\]
Note that for all $f \in C(X)$, 
\[
|\psi_{n_k}(f\circ T)-\psi_{n_k}(f)|\leq \int_X\left|\frac{1}{n_k}(f\circ T^{n_k}-f)\right| \, d\mu \leq \frac{2 \|f\|}{n_k}.
\]
Then for all $f \in C(X)$, letting $k$ tend to infinity gives 
\[
\psi(f\circ T)=\psi(f),
\]
which ensures the invariance of $\psi$. Note that $\psi$ is positive (i.e., if $f \geq 0$ then $\psi(f) \geq 0$), since that property is inherited from $\mu$. Also, for all $f\in \mathcal{C}$, we have $\psi(f)=0$, since $\psi_n(f)=0$ for all $n\in\mathbb N$ (where we have used the hypothesis that $\C$ is $T$-invariant). 

Since $X$ is a compact metrizable space and $\psi$  is a bounded positive linear functional on $C(X)^*$, by the Riesz-Markov Theorem (see \cite[p.458]{royden2018real}), there is a Borel measure $\hat{\mu}$ such that for all $f \in C(X)$, we have
\[
\psi(f)=\int_X f \, d\hat{\mu}. %, \quad  \forall f\in C(X).
\]
Also, $\hat{\mu}$ is $T$-invariant because
\[
\int_X f\circ T \, d\hat{\mu}=\psi(f\circ T) =\psi(f)=\int_X f \, d\hat{\mu}.
\]
Additionally, we have $\psi(\mathbf{1})=1$ since $\psi_n(\mathbf{1})=1$ for all $n\in\mathbb N$. Hence $\hat{\mu} \in M(X,T)$. Finally, for any $f\in \mathcal{C}$, we have $\hat{\mu}(f)=\psi(f)=0$, and therefore $\hat{\mu}\in M_\C(X,T)$.
\end{proof}

We are now prepared to prove our main existence result, which we recall for the reader's convenience.
\existence*

\begin{proof}
First, let us prove that for any $\mu \in M(X,T)$, we have $\mu|_\mathcal{C} \equiv 0$ if and only if $\mu|_H \equiv 0$. Since $\mathcal{C}\subseteq H$, it is immediate that $\mu|_H \equiv 0$ implies $\mu|_\mathcal{C} \equiv 0$. Now for the reverse direction, suppose $\mu|_\C \equiv 0$. Then $\ker(\mu)$ is a closed $T$-invariant subspace containing $\C$, and hence $H \subset \ker(\mu)$ (by the minimality of $H$). Therefore $\mu|_H \equiv 0$. 
%observe that for any positive bounded linear functional $\mu$ in $ C(X)^*$ that is $T$-invariant and equal to $0$ on $\mathcal{C}$, its kernel, $\mathrm{ker}(\mu):=\{f\in C(X): \mu(f)=0\}$, is closed and $\mathcal{C}\subseteq\mathrm{ker}(\mu)$. Since for any $f\in\mathrm{ker}(\mu)$, we have $f\circ T\in\mathrm{ker}(\mu)$. Therefore, $H\subseteq\mathrm{ker}(\mu)$ as well because $H$ is the smallest $T$-invariant linear subspace of $C(X)$ containing $\C$, which implies that $\mu|_H=0$.

Now let
\[
P=\{f\in C(X): f(x)\geq 0,\ \forall x\in X\},
\]
which is the non-negative cone in $C(X)$, and let $\mathrm{int}(P)$ denote its interior. Note that $g \in \mathrm{int}(P)$ if and only if %there exists $\epsilon >0$ such that $g(x) \geq \epsilon$ for all $x \in X$, and hence $g \in \mathrm{int}(P)$ if and only if 
$g(x) >0$ for all $x \in X$. To prove the equivalence of (i) and (ii) in the theorem, let us show that
\begin{equation} \label{Eqn:Juniper}
M_\C(X,T)\neq\varnothing\iff H\cap\mathrm{int}(P)=\varnothing.
\end{equation}
First, suppose there exists some measure $\mu$ in $M_\C(X,T)$. Then $\mu|_\C \equiv 0$, and by the previous paragraph $\mu|_H \equiv 0$. Suppose for contradiction that the intersection of $H$ with $\mathrm{int}(P)$ is non-empty. Then there exists $f \in H \cap P$ and  $t < 0$ such that $f+t$ is still in $P$. Next observe that $\mu(f+t) \geq 0$ (since $\mu$ is a probability measure), and on the other hand, we have $\mu(f+t) = \mu(f) + t = t < 0$ (since $f \in H$ and $\mu|_H \equiv 0$), a contradiction. Hence $H \cap \mathrm{int}(P) = \varnothing$.

To prove the reverse direction in (\ref{Eqn:Juniper}), we suppose that the intersection of $H$ with $\mathrm{int}(P)$ is empty. Recall that $\mathbf{1} \in C(X)$ denotes the constant function equal to one, and note that since $\mathbf{1} \in \mathrm{int}(P)$, we have that $\mathbf{1} \notin H$. Consider the subspace $Z = \mathrm{span}(H,\mathbf{1})$. Then define $\ell : Z \to \R$ by setting $\ell(\lambda+f)=\lambda$ for all $f\in H$ and $\lambda\in\mathbb R$ (and note that $\ell$ is well-defined on $Z$, since $H \cap \{\lambda \cdot \mathbf{1} : \lambda \in \R\} = \{0\}$). 
%Moreover, $\ell$ is a positive linear functional on $Z$, since for any $y=\lambda+f\geq 0$ in $\mathrm{span}(H,\mathbf{1})$, we have $\lambda \geq 0$ (if not, as $H\cap\mathrm{int}(P)=\varnothing$, there must be some $x\in X$ such that $f(x)+\lambda<0$), which implies that $\ell(y)=\ell(\lambda+f)\geq 0$. 
First we claim that $\ell$ is a positive linear functional on $Z$. The linearity is clear from the definition. For positivity, suppose $f + \lambda \geq 0$ for some $f \in H$ and $\lambda \in \R$. If $\lambda < 0$, then $f \geq -\lambda >0$, which gives that $f \in H \cap \mathrm{int}(P)$, a contradiction. Hence $\lambda \geq 0$, and then $\ell(f+\lambda) = \lambda \geq 0$, which shows that $\ell$ is positive on $Z$.
Next we claim that $\ell$ is bounded by a constant times the positively homogeneous and subadditive functional $\rho(g) = \|g\|$. To that end, define the distance function $\mathrm{dist}(\cdot,H) : C(X)\to\R$ by
\[
\mathrm{dist}(g,H):=\inf_{f\in H}\|g-f\|.
\]
Since $H$ is closed and $\mathbf{1} \notin H$, we have $d_0 := \mathrm{dist}(\mathbf{1},H) >0$. Next observe that for any $\lambda \neq 0$ and $f \in H$, we have
  \begin{equation*}
     \|f+\lambda\| = \|\lambda - (-f)\| = |\lambda|\cdot \| \mathbf{1} - (-f)/\lambda \| \geq |\lambda| \cdot d_0.
    \end{equation*}
    Then we obtain
    \begin{equation*}
        \rho(f+\lambda)=\|f+\lambda\| \geq |\lambda| \cdot d_0,
    \end{equation*}
    and thus $|\ell(f+ \lambda)| \leq \frac{1}{d_0} \|f+\lambda\|$ for all $f + \lambda \in Z$.
    
    By the Hahn-Banach theorem for positive linear functionals (see \cite[Theorem 2.3.7]{buhler2018functional}), there is a positive linear functional $\tilde{\ell}$ on $ C(X)$ such that $\tilde\ell|_H=0$, $\tilde\ell(\mathbf{1})=1$, and $|\tilde{\ell}(g)| \leq\frac{1}{d_0}\|g\|$ for any $g\in C(X)$. Then by the Riesz Representation Theorem for the dual of $C(X)$ (see \cite[Corollary 14.15]{aliprantis2006infinite}), there exists a signed Borel measure $\tilde{\mu}$ such that for any $g\in C(X)$, we have
    \[
    \tilde{\ell}(g)=\int g \, d\tilde{\mu}.
    \]
    By the properties of $\tilde{\ell}$ established above, we get that $\tilde{\mu}|_H \equiv 0$ and $\tilde{\mu}(\mathbf{1})=1$. Moreover, since $\tilde{\ell}$ is positive, we have that $\tilde{\mu}$ is a probability measure. Then by Lemma \ref{Lem:Existence}, there exists a measure $\mu^*\in M_H(X,T)$, which implies $\mu^*\in M_\C(X,T)$. Therefore, $M_\C(X,T)$ is non-empty, which concludes the proof of (\ref{Eqn:Juniper}).

    Finally, suppose that $M_\C(X,T)$ is non-empty, and let us show that the solution set $M_\C^*(X,T;\phi)$ is non-empty. As $\phi\in C(X)$ is continuous, the map $F: M_\C(X,T) \to \R$ defined by $F(\mu) = \int \phi \, d\mu$ is continuous (see \cite{jenkinson2006ergodic}). %That is, if $\mu_n\to\mu$ in $M_\C(X,T)$ with weak$^*$ topology, then
    %\[
    %\int\phi\ d\mu\geq\limsup_{n\to\infty}\int\phi\ d\mu_n.
    %\]
    Then, since $M_{\C}(X,T)$ is compact (by \Cref{Prop:M_C property}), by the Extreme Value Theorem, $F$ is bounded above, and moreover the supremum of $F$ is attained. 
\end{proof}

\begin{remark}
In Examples \ref{Example1}, \ref{Example2}, \ref{Example4} and \ref{Example5}, one may easily check that the constraint set $\C$ itself is closed and $T$-invariant, and therefore \Cref{Thm:Existence} yields that $M_\C(X,T)$ is non-empty.
\end{remark}

\begin{example}
\label{Exp:Invfinitemany}
Suppose $X$ is a non-empty compact metrizable set and $T_1 : X \to X$ and $T_2 : X \to X$ are continuous commuting transformations ($T_1\circ T_2=T_2\circ T_1$). For $i=1,2$, let
\[
\mathcal{C}_i=\{f-f\circ T_i: f\in C(X)\}.
\]
Then, by \Cref{Prop:M_C property} (iv), we have $M_{\C_2}(X,T_1)=M_{\C_1\cup\C_2}(X,\mathrm{Id})=M_{\C_1}(X,\mathrm{Id})\cap M_{\C_2}(X,\mathrm{Id})$, which is the set of measures that are both $T_1$- and $T_2$-invariant. Since $T_1\circ T_2=T_2\circ T_1$, we observe that
\[
(f-f\circ T_2)\circ T_1=f\circ T_1-(f\circ T_1)\circ T_2\in \mathcal{C}_2,\quad\text{for any}\ f\in C(X),
\]
and therefore $\C_2$ is $T_1$-invariant. Since there exists a measure $\mu$ equal to $0$ on $\C_2$ (given by any $\mu\in M(X,T_2)$), \Cref{Thm:Existence} yields that $M_{\C_2}(X,T_1)$ is non-empty. We note that the non-emptiness of $M_{\C_2}(X,T_1)$ can also be obtained from the amenability of the semigroup action generated by the transformations.

%Furthermore, let $\C^{(n-1)}=\cup_{k=1}^{n-1}\C_k$, we have that $M_{\C^{(n-1)}}(X,T_n)$ is non-empty by induction.
\end{example}

% \begin{remark}
%     Let $\Gamma=\mathbb Z^n$. Let $\phi$ be a topological action of $\Gamma$ on $X$, that is, let $\gamma=(a_0,a_1,\dots,a_{n-1})$ be any element in $\Gamma$,  $\phi:\Gamma\times X\to X$, $(\gamma,x)\mapsto\phi_\gamma(x)=T_1^{a_0}\circ T_2^{a_1}\circ\dots\circ T_n^{a_{n-1}}(x)$ is a continuous map such that $\phi_{(0,0,\dots,0)}=\mathrm{Id}$ and $\phi_{\gamma_1}\circ\phi_{\gamma_2}=\phi_{\gamma_1\gamma_2}$ for all $\gamma_1,\gamma_2\in\Gamma$. Define
%     \[
%     M_\phi=\{\mu\in M(X,T):\mu(\phi_\gamma^{-1}A)=\mu(A)\ \text{for all}\ \gamma\in\Gamma\ \text{and}\ A\in\mathcal{B}(X)\}.
%     \]
%     Notice that $M_\phi$ is the same as $M_{\C^{(n-1)}}(X,T_n)$ in Example \ref{Exp:Invfinitemany}. The nonemptiness of $M_\phi$ can also be obtained from the amenability of the semigroup action generated by the transformations.
% \end{remark}

\section{Uniqueness} \label{Sect:Uniqueness}
By Theorem \ref{Thm:Existence}, if $(X,T)$ is a topological dynamical system and $\C\subset C(X)$ is a constraint set such that $M_\C(X,T)$ is non-empty, then the optimal solution set $M^*_\mathcal{C}(X,T;\phi)$ is non-empty for any $\phi\in C(X)$. In general, the solution set $M^*_\C(X,T;\phi)$ may contain more than one element. An extreme example is when $\mathcal{C}=\{0\}$ and $\phi$ is a constant, in which case every measure $\mu\in M(X,T)$ maximizes $\int\phi \, d\mu$.
In this section, we are interested in the following uniqueness problem: under what conditions is there a unique solution to the constrained ergodic optimization problem (\ref{Eqn:EOWC})? In ergodic optimization, there are two common senses in which a set of objective functions is considered large: residual sets are large in a topological sense, and prevalent sets are large in the sense of measure theory. Here we give sufficient conditions under which the set of objective functions with a unique solution is large in both senses, generalizing results of \cite{jenkinson2006ergodic} for the standard (unconstrained) ergodic optimization problem. 

\subsection{Generic Uniqueness} 
We consider a function space $E\subset C(X)$, and we would like to determine whether the constrained ergodic optimization problem (\ref{Eqn:EOWC}) with a `typical' function $\phi \in E$ has a unique solution. Let
%and a `{\it residual}' (A residual set is the set contains a countable intersection of open dense subsets.) subset $E'\subset E$ such that for any function $\phi\in E'$, (\ref{Eqn:EOWC}) has a unique solution. Or in other words, the set
\[
U_\mathcal{C}(E)=\bigl\{\phi\in E: |M_\C^*(X,T;\phi)|=1\bigr\}.
\]
%is residual (so it is topologically large) in $E$. Such uniqueness property that is associated with a residual subset is called a `{\it generic}' property. 
Our first uniqueness result deals with the case when $M_\C(X,T)$ contains only finitely many extreme points, generalizing \cite[Proposition 3.1]{jenkinson2006ergodic}. We note that our proof follows essentially the same idea as the proof of \cite[Proposition 3.1]{jenkinson2006ergodic}.
\begin{proposition}
Let $(X,T)$ be a topological dynamical system and $\C \subset C(X)$. Let $E$ be a topological vector space that is densely and continuously embedded in $C(X)$. If  $M_\C(X,T)$ is non-empty and contains only finitely many extreme points, then $U_\mathcal{C}(E)$ is open and dense in $E$.  
\end{proposition}
\begin{proof}
Let $\{\mu_1,\dots,\mu_N\}$ be the extreme points in $M_\C(X,T)$, and define
\[
M_i=\left\{\phi\in E: \text{for all} \, \mu \in M_\C(X,T), \, \int \phi \, d\mu_i\ \geq \int\phi \, d\mu\ \right\},
\]
for all $1\leq i\leq N$. Since $M_{\C}^*(X,T,\phi)$ is always a face of $M_{\C}(X,T)$, if (\ref{Eqn:EOWC}) has a unique solution, then it must be extreme in $M_\C(X,T)$. So the complement of $U_\mathcal{C}(E)$ can be expressed as
\[
U_\mathcal{C}(E)^c=\bigcup_{i<j}M_i\cap M_j.
\]
In the following, we show that $M_i$ is closed and $M_i\cap M_j$ is hollow for each $i\neq j$. Suppose that $\{\phi_\alpha\}$ is a net in $M_i$, with $\phi_\alpha\to\phi$ in $E$. Then $\phi_\alpha\to\phi$ in $C(X)$, since $E$ is continuously embedded in $C(X)$. Then we have $\int\phi_\alpha \, d\mu_i\geq\int\phi_\alpha \, d\mu$ for any $\mu\in M_\C(X,T)$, and hence  $\int\phi \, d\mu_i\geq\int\phi \, d\mu$ for any $\mu\in M_\C(X,T)$. Thus $\phi\in M_i$, and we obtain that $M_i$ is closed.
Now we prove that each $M_i\cap M_j$ is hollow. Since $E$ is densely embedded in $C(X)$, for any $i<j$ there exists $g=g_{ij}\in E$ such that $\int g \, d\mu_i\neq\int g \, d\mu_j$. If $\phi\in M_i\cap M_j$, then for every $\varepsilon>0$ the function $\phi+\varepsilon g$ is not in $M_i\cap M_j$ since $\int(\phi+\varepsilon g) \, d\mu_i\neq\int(\phi+\varepsilon g) \, d\mu_j$. Hence $M_i\cap M_j$ has empty interior whenever $1\leq i<j\leq N$. Therefore, $U_\mathcal{C}(E)$ is open and dense in $E$.  
\end{proof}

The assumption that $M_\C(X,T)$ has only finitely many extreme points is very strong. In the theorem below we consider the general situation in which no assumption is placed on the extreme points. The trade-off for this more general hypothesis is that we can only conclude that $U_{\C}(E)$ is residual (instead of open and dense). % So we strive to get a uniqueness result for a more general $T$. 
% For a specific property $\mathcal{P}$, we want to find a \textit{residual} set $E^\prime$ such that every $f\in E^\prime$ has the property $\mathcal{P}$. A residual set is the set contains a countable intersection of open dense subsets. We say that $\mathcal{P}$ is a \textit{generic} property if there is some residual set $E^\prime$ such that every member of $E^\prime$ has the property $\mathcal{P}$.
To prove this theorem, we follow the idea of the proof of \cite[Theorem 1]{jenkinson2025typical} and find it helpful to refer to {\it Fort's Theorem}, which we now describe. A set-valued map $\varphi:X\to 2^Y$ is called a `{\it correspondence},' and we denote it by $\varphi: X\twoheadrightarrow Y$. Recall that a subset of $X$ is `\emph{residual}' if it is a countable intersection of open and dense subsets of $X$. We first introduce the (semi-)continuity of correspondences.

\begin{definition}[\cite{aliprantis2006infinite}, Definition 17.2]
\label{Def:LowUpCts}
    Suppose $X$ and $Y$ are topological spaces. A correspondence $\varphi:X\twoheadrightarrow Y$ is 
    \begin{enumerate}[(i)]
        \item {\it lower semi-continuous (lsc)} at $x\in X$ if for every open set $\mathcal{O}$ in $Y$ such that $\varphi(x)\cap\mathcal{O}\neq\varnothing$, there is a neighborhood $\mathcal{N}$ of $x$ such that if $z\in\mathcal{N}$ then $\varphi(z)\cap\mathcal{O}\neq\varnothing$; 
        \item {\it upper semi-continuous (usc)} at $x\in X$ if for every open set $\mathcal{O}$ in $Y$ such that $\varphi(x)\subset\mathcal{O}$, there is a neighborhood $\mathcal{N}$ of $x$ such that if $z\in \mathcal{N}$ then $\varphi(z)\subset\mathcal{O}$. 
    \end{enumerate}
    Additionally, we say $\varphi$ is lower semi-continuous on $X$ if it is lsc at every $x\in X$, and $\varphi$ is upper semi-continuous on $X$ if it is usc at every $x\in X$.
\end{definition}

\begin{theorem}[Fort \cite{fort1951points}]
\label{Thm:Fort}
If $\varphi: X\twoheadrightarrow Y$ is an upper (lower) semi-continuous correspondence from a topological space $(X,\tau)$ into the non-empty compact subsets of a metric space
$(Y,d)$, then $F$ is both upper and lower semi-continuous at the points of a residual subset of $X$. Furthermore, if $(X,\tau)$ is a Baire space, then $F$ is continuous at the points of a dense $G_\delta$ subset of $X$.
\end{theorem}

\begin{theorem}
\label{Thm:Uniqueness}
Let $(X,T)$ be a topological vector space and $\mathcal{C}\subset C(X)$, and suppose that $M_\C(X,T)\neq\varnothing$. Let $E$ be a topological vector space which is densely and continuously embedded in $C(X)$. Then $U_\mathcal{C}(E)$ is a residual subset of $E$. Moreover, if $E$ is a Baire space, then $U_\mathcal{C}(E)$ is dense in $E$.
\end{theorem}

\begin{proof}
Since $M_{\C}(X,T)$ is non-empty, Proposition \ref{Prop:BasicM*} guarantees that $M^*_{\C}(X,T; \phi)$ is non-empty for all $\phi \in C(X)$. 
We define the correspondence $c: E\twoheadrightarrow  M_\C(X,T)$ by $c(\phi)=M^*_\C(X,T;\phi)$. Note that $c(\phi)$ is non-empty for all $\phi\in E$ by Proposition \ref{Prop:BasicM*} (using that $M_\C(X,T)$ is non-empty by hypothesis).  By \cite[Theorem 17.11]{aliprantis2006infinite}, the upper semi-continuity of $c$ is equivalent to the closedness of the graph $\mathrm{Gr}(c):=\{(\phi,\mu):\phi\in E,\mu\in c(\phi)\}$ in $E\times M_\C(X,T)$. To prove that $\mathrm{Gr}(c)$ is closed, it suffices to show its closure, $\overline{\mathrm{Gr}(c)}$, equals $\mathrm{Gr}(c)$. Suppose $\{(\phi_\alpha,\mu_\alpha)\}$ is a net in $\mathrm{Gr}(c)$ that converges to $(\phi_0,\mu_0)$ in $E \times M_\C(X,T)$. Since $E$ is continuously embedded in $C(X)$, we also have that $\phi_{\alpha}$ converges to $\phi_0$ in $C(X)$. To establish that $\mathrm{Gr}(c)$ is closed, let us show that $\mu_0 \in c(\phi_0)$. Let $\mu \in M_\C(X,T)$. For all $\alpha$ we have $\mu_\alpha\in c(\phi_\alpha)$, and thus 
\begin{equation} \label{Eqn:USMNT}
\int\phi_\alpha \, d\mu_\alpha\geq\int\phi_\alpha \, d\mu.
\end{equation}
As $(\phi_\alpha,\mu_\alpha)\to(\phi_0,\mu_0)$ and $|\int(\phi_\alpha-\phi_0) \, d\mu_\alpha|\leq\int\|\phi_\alpha-\phi_0\| \, d\mu_\alpha\leq \|\phi_\alpha-\phi_0\|\to 0$, we have $\int\phi_\alpha \, d\mu_\alpha=\int(\phi_\alpha-\phi_0) \, d\mu_\alpha+\int\phi_0 \, d\mu_\alpha\to\int\phi_0 \, d\mu_0$. Using this fact and the fact that $\phi_\alpha \to \phi_0$ in (\ref{Eqn:USMNT}) gives $\int\phi_0 \, d\mu_0\geq \int\phi_0 \, d\mu$. Since $\mu \in M_\C(X,T)$ was arbitrary, we obtain that $\mu_0\in c(\phi_0)$. This establishes the closedness of $\mathrm{Gr}(c)$, and hence we conclude that $c$ is upper semi-continuous.% is satisfied. 

By Fort's Theorem (Theorem \ref{Thm:Fort} above), it suffices to show that the set of continuity points for $c$ is precisely $U_\C(E)$. By the previous paragraph $c$ is upper semi-continuous. To complete the proof, let us now show that $c$ is lower semi-continuous at $\phi$ if and only if $\phi\in U_\C(E)$.  %Since $c$ is upper semi-continuous, it suffices to show that $c$ is lower semi-continuous at $\phi$ if and only if $\phi\in U_\C(E)$. 

If $\phi\in U_\C(E)$, then $c(\phi)$ is a singleton. For any open set $\mathcal{O}\subset M_\C(X,T)$ such that $c(\phi)\cap \mathcal{O}\neq\varnothing$, we have that the singleton set $c(\phi)$ satisfies $c(\phi) \subset \mathcal{O}$, and hence $\mathcal{O}$ is a neighborhood of $c(\phi)$. Then, as $c$ is upper semi-continuous, there exists a neighborhood $\mathcal{N}$ of $\phi$ such that if $z\in \mathcal{N}$ then $c(z)\subset\mathcal{O}$, and thus $c(z)\cap\mathcal{O}\neq\varnothing$, as $c(z) = M_{\C}^*(X,T;z)$ is non-empty. Therefore, $c$ is lower semi-continuous at $\phi$.

If $\phi\notin U_\C(E)$, then $c(\phi)$ contains at least two distinct measures $\mu_1 \neq \mu_2$. Then by the density of $E$ in $C(X)$, there exists $g\in E$ such that $\int g \, d\mu_1<\int g \, d\mu_2$. Let $\mathcal{O}_1=\{\mu\in M_\C(X,T):\int g \, d\mu>\int g \, d\mu_1\}$, and notice that $\mathcal{O}_1$ is an open set that contains $\mu_2$ but does not contain $\mu_1$. Hence $\mu_2\in c(\phi)\cap\mathcal{O}_1$. We claim that for any neighborhood $\mathcal{N}$ of $\phi$, there exists $z\in \mathcal{N}$ such that $z\notin c^{-1}(\mathcal{O}_1)=\{h\in E: c(h)\cap\mathcal{O}_1\neq\varnothing\}$. Indeed, let $\mathcal{N}$ be a neighborhood of $\phi$. Choose $z=\phi-\varepsilon g$,  where $\varepsilon>0$ is sufficiently small to ensure that $z \in \mathcal{N}$. Let $\mu \in \mathcal{O}_1$. Since $\mu_1 \in c(\phi)$, we have $\int\phi \, d\mu_1\geq\int\phi \, d\mu$, and since $\mu\in\mathcal{O}_1$, we have $\int g \, d\mu_1<\int g \, d\mu$. Hence we obtain that for any $\mu\in\mathcal{O}_1$
\[
    \int(\phi-\varepsilon g) \, d\mu_1>\int(\phi-\varepsilon g) \, d\mu.
\]
Therefore, $c(z)\cap\mathcal{O}_1=\varnothing$, and we conclude that $c$ is not lower semi-continuous at $\phi$, which finishes the proof.
\end{proof}  

\begin{remark} \label{Rmk:Uniqueness1}
    If we take $E=C(X)$ in Theorem \ref{Thm:Uniqueness}, then we obtain that the set $U_\C=\{\phi\in C(X): |M_\C^*(X,T;\phi)|=1\}$ is residual, as claimed in Theorem \ref{Thm:UniquenessIntro}. Moreover, since $C(X)$ is a Baire space, the set $U_\C$ is dense in $C(X)$.
\end{remark}

\subsection{Prevalent uniqueness} \label{Sect:Prevalence}
In the previous section, we give conditions under which the set of functions with a unique maximizing measure is topologically large. 
In this section, we show that under the same conditions the set of functions with a unique maximizing measure is measure-theoretically large. More precisely, we show that the set $U_{\C}(E)$ is prevalent, a notion introduced in \cite{hunt1992prevalence}. Let us briefly recall the definintion of prevalence. A subset $A$ of a topological group $G$ is said to be {\it Haar null} if there is a Borel set $B\supseteq A$ and a Borel probability measure $\mu$ on $G$ such that $\mu(gBh) = 0$ for every $g, h \in G$. A set $A\subset G$ is \emph{prevalent} if $G\setminus A$ is Haar null. See \cite{elekes2020haar} for more details.
%The prevalent uniqueness problem talks about when $E\subset C(X)$ has a {\it prevalent} subset of functions with unique maximizing measure. Or equivalently, when $U_\mathcal{C}(E)\subset E$ is prevalent? Unlike the residual set, the prevalent set is a probabilistically large subset. It was first introduced in \cite{hunt1992prevalence}.

The prevalent uniqueness for (unconstrained) ergodic optimization was first proved by Morris \cite{morris2021prevalent}. In the following, we give a prevalent uniqueness result for our constrained optimization problem (\ref{Eqn:EOWC}).

For notation, let $F_\C : C(X) \to \R$ be defined by
\[
F_\C(\phi)=\sup_{\mu\in M_\C(X,T)}\int\phi \, d\mu.
\]

Next, we introduce the notion of G\^{a}teaux differentiability of $F_\C$. Recall that a {\it Fr\'{e}chet space} is a completely metrizable locally convex vector space.  

\begin{definition}[G\^{a}teaux differentiability \cite{benyamini2000geometric}]
    Let $f$ be a function defined on an open set in a Banach space $X$ with values in a Banach space $Y$. The function $f$ is said to be {\it G\^{a}teaux differentiable} at a point $x_0$ if there is a bounded linear operator $T: X\to Y$ such that for all $u \in X$, 
    \begin{equation}
    \label{Eqn:Gateaux}
        Tu=\lim_{t\to 0}\frac{f(x_0+tu)-f(x_0)}{t}.
    \end{equation}
    %for every $u\in X$. 
    The (uniquely determined) operator $T$ is called the G\^{a}teaux derivative of $f$ at $x_0$.
\end{definition}

\begin{theorem}
\label{Thm:GdiffCphi}
Let $(X,T)$ be a topological vector space and $\mathcal{C}\subset C(X)$, and suppose that $M_\C(X,T)\neq\varnothing$.
If $E$ is a Fr\'{e}chet space that is densely and continuously embedded in $C(X)$, then $F_\C$ is G\^{a}teaux differentiable at $\phi$ if and only if $\phi\in U_\C(E)$. 
\end{theorem}
\begin{proof}  Since $M_{\C}(X,T)$ is non-empty, Proposition \ref{Prop:BasicM*} guarantees that $M^*_{\C}(X,T; \phi)$ is non-empty for all $\phi \in C(X)$. We proceed by establishing some preliminary inequalities. Let $\phi$ and $u$ be in $C(X)$, and let $t \neq 0$. For any $\mu_\phi\in M^*_\mathcal{C}(X,T;\phi)$ and $\mu_{\phi+tu}\in M^*_\mathcal{C}(X,T;\phi+tu)$, we have
\[
\begin{split}
    \int(\phi+tu) \, d\mu_\phi-\int\phi \, d\mu_\phi &\leq F_\C(\phi+tu)-F_\C(\phi)\\
    &\leq\int(\phi+tu) \, d\mu_{\phi+tu}-\int\phi \, d\mu_{\phi+tu}.
\end{split}
\]
Using the linearity of the integral and rearranging gives
\begin{equation}\label{Ineq:Gdiff}
    \begin{cases}
        \int u \, d\mu_\phi\leq\frac{F_\C(\phi+tu)-F_\C(\phi)}{t}\leq\int u \, d\mu_{\phi+tu} & \text{if}\ t>0\\[0.5em]
        \int u \, d\mu_\phi\geq\frac{F_\C(\phi+tu)-F_\C(\phi)}{t}\geq\int u \, d\mu_{\phi+tu} & \text{if}\ t<0\\
    \end{cases}.
\end{equation}
By taking the supremum over $M_\C^*(X,T;\phi)$ on the left hand side and taking the supremum over $M_\C^*(X,T;\phi+tu)$ on the right hand side of both cases in (\ref{Ineq:Gdiff}), we obtain
\begin{equation}\label{Ineq2:Gdiff}
    \begin{cases}
        \sup\limits_{\mu\in M^*_\mathcal{C}(X,T;\phi)}\int u \, d\mu\leq\frac{F_\C(\phi+tu)-F_\C(\phi)}{t}\leq\sup\limits_{\mu\in M^*_\mathcal{C}(X,T;\phi+tu)}\int u \, d\mu, &\text{if}\ t>0\\[1em]
        \sup\limits_{\mu\in M^*_\mathcal{C}(X,T;\phi)}\int u \, d\mu\geq\frac{F_\C(\phi+tu)-F_\C(\phi)}{t}\geq\sup\limits_{\mu\in M^*_\mathcal{C}(X,T;\phi+tu)}\int u \, d\mu, &\text{if}\ t<0
    \end{cases}.
\end{equation}
Subtracting $\sup_{\mu\in M^*_\mathcal{C}(X,T;\phi)}\int u \, d\mu$ from (\ref{Ineq2:Gdiff}) gives
\begin{equation}\label{Ineq3:Gdiff}
\begin{split}
    0&\leq\left|\frac{F_\C(\phi+tu)-F_\C(\phi)}{t}-\sup_{\mu\in M^*_\mathcal{C}(X,T;\phi)}\int u \, d\mu\right|\\
    &\hspace{5em}\leq \left|\sup_{\mu\in M^*_\mathcal{C}(X,T;\phi+tu)}\int u \, d\mu-\sup_{\mu\in M^*_\mathcal{C}(X,T;\phi)}\int u \, d\mu\right|.
\end{split}
\end{equation}
% We first prove that $\phi\in U_\C(E)$ implies that $F_\C$ is G\^{a}teaux differentiable at $\phi$. 

Let us now prove that if $\phi\in U_\C(E)$ then $F_\C$ is G\^{a}teaux differentiable at $\phi$. To that end, let $\phi$ be in $U_\C(E)$. Since $\phi\in U_\C(E)$, there exists $\mu^*\in M_\C(X,T)$ such that $M_\C^*(X,T;\phi)=\{\mu^*\}$. 

\begin{claim}\label{Claim:Gdiff}
    For any $u\in E$, we have
    \begin{equation}
\label{Lim:Gdiff}
\lim_{t\to 0}\sup_{\mu\in M^*_\mathcal{C}(X,T;\phi+tu)}\int u \, d\mu=\int u \, d\mu^*.    
\end{equation}
\end{claim}
\begin{proof}[Proof of \Cref{Claim:Gdiff}]
Select a sequence $\{t_n\}_{n=1}^\infty$ such that $t_n\to 0$ as $n\to\infty$. For each $n \geq 1$, since $M^*_{\C}(X,T; \phi + t_n u)$ is non-empty,
%and
% \[
% \limsup_{t\to 0}\left[\sup_{\mu\in M_\C^*(X,T;\phi+tu)}\int u \, d\mu\right]=\lim_{n\to\infty}\left[\sup_{\mu\in M_\C^*(X,T;\phi+t_nu)}\int u \, d\mu\right].
% \]
there is a $\mu_n$ that satisfies 
\begin{equation} \label{Eqn:CCD}
\mu_n\in\argmax\limits_{\mu\in M^*_\mathcal{C}(X,T;\phi+t_n u)}\int u \, d\mu.
\end{equation}

Since $M_\C(X,T)$ is compact in the weak$^*$ topology, for any subsequence $\{\mu_{n_k}\}_{k=1}^\infty$ of $\{\mu_n\}_{n=1}^\infty$ we can find a further subsequence $\{\mu_{n_{k_j}}\}_{j=1}^\infty$ such that $\{\mu_{n_{k_j}}\}_{j=1}^\infty$ converges to some measure $\mu_0\in M_\C(X,T)$ as $j\to\infty$. Now let us show that $\mu_0\in M_\C^*(X,T;\phi)$ (and hence $\mu_0 = \mu^*$). By (\ref{Eqn:CCD}), for any $\mu\in M_\C(X,T)$, we have
\begin{equation}
\label{Eqn:Gdiff}
\int(\phi+t_{n_{k_j}}u) \, d\mu_{n_{k_j}}\geq \int(\phi+t_{n_{k_j}}u) \, d\mu.
\end{equation}
%Letting $k\to\infty$ in (\ref{Eqn:Gdiff}) gives $\lim_{k\to\infty}\int t_{n_k}u\, d\mu_{n_k}=0$ because
Next observe that $\lim_{j\to\infty}\int t_{n_{k_j}}u\, d\mu_{n_{k_j}}=0$ since
\[
\lim_{j\to\infty}\left|\int(t_{n_{k_j}}u) \, d\mu_{n_{k_j}}\right|\leq \lim_{j\to\infty}\int|t_{n_{k_j}}u| \, d\mu_{n_{k_j}}\leq \lim_{j\to\infty}\|t_{n_{k_j}}u\|= \lim_{j\to\infty} |t_{n_{k_j}}| \cdot \|u\| = 0.
\]
Therefore, taking the limit as $j$ tends to infinity in (\ref{Eqn:Gdiff}) yields $\int\phi \, d\mu_0\geq\int\phi \, d\mu$. Since this inequality holds for any $\mu\in M_\C(X,T)$, we conclude that $\mu_0\in M^*_\C(X,T;\phi)$, and hence $\mu_0=\mu^*$. Thus we have shown that every subsequence $\{\mu_{n_k}\}_{k=1}^\infty$ of $\{\mu_n\}_{n=1}^\infty$ has a further subsequence $\{\mu_{n_{k_j}}\}_{j=1}^\infty$ that converges to $\mu^*$, and therefore we see that $\mu_n\to\mu^*$. 

% If not, there is a subsequence $\{\mu_{n_k}\}_{k=1}^\infty$ such that its convergent subsubsequence $\{\mu_{n_{k_j}}\}_{j=1}^\infty$ does not converges to $\mu^*$, this contradicts the fact that every convergent subsequence converges to $\mu^*$.

By the definition of $\mu_n$ in (\ref{Eqn:CCD}), we have
%for each $\mu_{n_k}$ in the convergent sequence $\{\mu_{n_k}\}$,  
\begin{equation}\label{Eqn:sequ_n}
    \sup_{\mu\in M^*_\mathcal{C}(X,T;\phi+t_n u)}\int u \, d\mu= \int u \, d\mu_{n}.
\end{equation}
% Note that the above inequality is satisfied for every convergent subsequence $\{\mu_{n_k}\}$ of $\{\mu_n\}$. Since $\phi$ is in $U_\C(E)$, there exists $\mu^*$ such that $M_\C^*(X,T;\phi)=\{\mu^*\}$ and every convergent subsequence $\{\mu_{n_k}\}$ converges to the same $\mu^*$. 
Thus, if we let $n\to\infty$ in (\ref{Eqn:sequ_n}), then 
\[
    \lim_{n\to \infty}\left[\sup_{\mu\in M^*_\mathcal{C}(X,T;\phi+t_{n}u)}\int u \, d\mu\right]=\lim_{n\to\infty}\int u\, d\mu_n=\int u\, d\mu^*,
    %=\limsup_{t\to 0}\left[\sup_{\mu\in M^*_\mathcal{C}(X,T;\phi+tu)}\int u \, d\mu\right]\leq \int u \, d\mu^*.
    %=\sup_{\mu\in M_\C^*(X,T;\phi)}\int u \, d\mu.
\]
where we have used the convergence $\mu_n \to \mu^*$ (established in the previous paragraph).
% Similarly, we can select a sequence $\{t_m\}$ such that $t_m\to 0$ when $m\to \infty$, and 
% \[
% \liminf_{t\to 0}\left[\sup_{\mu\in M_\C^*(X,T;\phi+tu)}\int u \, d\mu\right]=\lim_{m\to\infty}\left[\sup_{\mu\in M_\C^*(X,T;\phi+t_m u)}\int u \, d\mu\right].
% \]
% Let $\mu_m$ be defined same as $\mu_n$ in (\ref{Eqn:CCD}), then 
% %there is a convergent subsequence $\{\mu_{m_l}\}$ satisfies
% \[
% \sup_{\mu\in M^*_\mathcal{C}(X,T;\phi+t_{m} u)}\int u \, d\mu\geq \int u \, d\mu_{m}.
% \]
% Note that $\{\mu_{m}\}$ converges to the same $\mu^*\in M_\C^*(X,T;\phi)$, letting $m\to\infty$ gives 
% \[
% \liminf_{t\to 0}\left[\sup_{\mu\in M^*_\mathcal{C}(X,T;\phi+tu)}\int u \, d\mu\right]\geq \int u \, d\mu^*
%=\sup_{\mu\in M_\C^*(X,T;\phi)}\int u \, d\mu.
%\]
% \[
% \begin{split}
%     &\liminf_{k\to \infty}\left[\sup_{\mu\in M^*_\mathcal{C}(X,T;\phi+t_{n_k}u)}\int u \, d\mu\right]\\
%     =&\liminf_{t\to 0}\left[\sup_{\mu\in M^*_\mathcal{C}(X,T;\phi+tu)}\int u \, d\mu\right]\geq \int u \, d\mu^*=\sup_{\mu\in M_\C^*(X,T;\phi)}\int u \, d\mu.
% \end{split}
% \]
% Therefore, as $t\to 0$, we have
% \[
% \sup_{\mu\in M^*_\mathcal{C}(X,T;\phi+tu)}\int u \, d\mu\to\sup_{\mu\in M^*_\mathcal{C}(X,T;\phi)}\int u \, d\mu,
% \]

Since the above limit is satisfied for every sequence $\{t_n\}_{n=1}^\infty$ with the property $t_n\to 0$ as $n\to\infty$, we conclude that the limit (\ref{Lim:Gdiff}) exists and is equal to $\int u\, d\mu^*$. This completes the proof of \Cref{Claim:Gdiff}.
\end{proof}

Next, for any $\phi\in C(X)$, we define the operator $T$ by
\[
T(u):=\int u \, d\mu^*.
\]
By (\ref{Ineq3:Gdiff}) and \Cref{Claim:Gdiff}, if $\phi\in U_\C(E)$, then $\lim_{t\to 0}\frac{F_\C(\phi+tu)-F_\C(\phi)}{t}=\int u\, d\mu^* = T(u)$. % by the squeeze theorem and it is equal to that $T(u)$ we have defined above. 
%We claim that when $\phi\in U_\C(E)$ with $M_\C^*(X,T;\phi)=\{\mu^*\}$, $T$ is a linear and bounded operator.  %Since the limit (\ref{Lim:Gdiff}) is equal to $\sup_{\mu\in M_\C^*(X,T;\phi)}\int u\, d\mu=\int u\, d\mu_0$, we have $T(u)=\int u \, d\mu_0$. 
The linearity of $T$ is given by the linearity of integration. Also, $T$ is bounded, since $|T(u)|\leq \int |u|\, d\mu^*\leq \|u\|$. Therefore, we conclude that $F_\C$ is G\^{a}teaux differentiable at $\phi$. 

Next let us prove that if $F_\C$ is G\^{a}teaux differentiable at $\phi$, then $\phi\in U_\C(E)$. This part is similar to the proof of \cite[Theorem 1]{morris2021prevalent}, and we write the proof here for completeness. Suppose $\phi\notin U_\C(E)$, and let us show that $\lim_{t\to 0}\frac{F_\C(\phi+tu)-F_\C(\phi)}{t}$ does not exist for some $u\in E$. Since $M^*_\mathcal{C}(X,T;\phi)$ is not a singleton, there are at least two distinct elements $\mu_1^*,\mu_2^*$ in $M_\C^*(X,T;\phi)$. Since $E$ is dense in $C(X)$, there exists a function $u\in E$ such that $\int u \, d\mu_1^*>\int u \, d\mu_2^*$. 

% Let us denote $\mu_n$ by a member of $M_\C^*(X,T;\phi+\frac{1}{n}u)$ that maximizes $\int u\, d\mu$, that is,
% \[
% \mu_n\in\argmax_{\mu\in M^*_\mathcal{C}(X,T;\phi+\frac{1}{n} u)}\int u \, d\mu.
% \]
% By compactness of $M_\C(X,T)$, we can select a subsequence $\{\mu_{n_k}\}_{k=1}^\infty$ such that $\mu_{n_k}\to\mu_2^*$. 

By (\ref{Ineq2:Gdiff}), we have
\[
\liminf_{t\to 0^+}\frac{F_\C(\phi+tu)-F_\C(\phi)}{t}\geq\sup_{\mu\in M^*_\mathcal{C}(X,T;\phi)}\int u \, d\mu\geq\int u \, d\mu_1^*.
\]
By replacing $u$ with $-u$, we have
\[
\liminf_{t\to 0^+}\frac{F_\C(\phi+t(-u))-F_\C(\phi)}{t}\geq\sup_{\mu\in M_\C^*(X,T;\phi)}\int (-u)\, d\mu.
\]
Letting $s = -t$, and using the previous inequality, we see that
\[
\begin{split}
    -\limsup_{s\to 0^-}\frac{F_\C(\phi+su)-F_\C(\phi)}{s}&=\liminf_{(-t)\to 0^-}\left[-\frac{F_\C(\phi+(-t)u)-F_\C(\phi)}{-t}\right]\\
    &=\liminf_{t\to 0^+}\frac{F_\C(\phi+(-t)u)-F_\C(\phi)}{t}\\
    &\geq -\inf_{\mu\in M_\C^*(X,T;\phi)}\int u\, d\mu.
\end{split}
\]
Rewriting this inequality, we obtain 
\[
    \limsup_{t\to 0^-}\frac{F_\C(\phi+tu)-F_\C(\phi)}{t}\leq \inf_{\mu\in M_\C^*(X,T;\phi)}\int u\, d\mu\leq\int u\, d\mu^*_2.
\]
By our choice of $u$ and the preceding inequalities, we conclude that 
\begin{equation*}
    \liminf_{t\to 0^-}\frac{F_\C(\phi+tu)-F_\C(\phi)}{t} \geq \int u \, d\mu_1^* > \int u \, d\mu_2^* \geq \limsup_{t\to 0^+}\frac{F_\C(\phi+tu)-F_\C(\phi)}{t},
\end{equation*}
and hence $\lim_{t\to 0}\frac{F_\C(\phi+tu)-F_\C(\phi)}{t}$ does not exist. Thus, $F_\C$ is not G\^{a}teaux differentiable at $\phi$, which completes the proof.
\end{proof}

\begin{theorem}
\label{Thm:PrevalentUniq}
Let $(X,T)$ be a topological vector space and $\mathcal{C}\subset C(X)$, and suppose that $M_\C(X,T)\neq\varnothing$.
If $E$ is a separable Fr\'{e}chet space that is densely and continuously embedded in $ C(X)$, then $U_\C(E)$ is a prevalent subset of $E$.  
\end{theorem}
\begin{proof}
Since $M_{\C}(X,T)$ is non-empty, Proposition \ref{Prop:BasicM*} guarantees that $M^*_{\C}(X,T; \phi)$ is non-empty for all $\phi \in C(X)$. First, we claim that the map $F_\C:  C(X)\to\mathbb R$ is 1-Lipschitz. For any $f$ and $g$ in $C(X)$, let $\mu_f$ and $\mu_g$ be elements of $M_\C^*(X,T;f)$ and $M_\C^*(X,T;g)$, respectively. Then we have 
\begin{equation*}
    F_\C(f)-F_\C(g)
    \leq \int f \, d\mu_f-\int g \, d\mu_f=\int(f-g) \, d\mu_f\leq \|f-g\|,
\end{equation*}
and
\begin{equation*}
    F_\C(g)-F_\C(f)\leq \int g \, d\mu_g-\int f \, d\mu_g=\int(g-f) \, d\mu_g\leq \|g-f\|,
\end{equation*}
which implies $|F_\C(f)-F_\C(g)|\leq\|f-g\|$. Hence $F_\C(\cdot)$ is $1$-Lipschitz on $C(X)$.

Since $E$ is continuously embedded in $ C(X)$, the supremum norm $\|\cdot\|$ is a continuous semi-norm on $E$, so $F_\C$ is Lipschitz on $E$. By Christensen \cite{christensen1973measure}, a Lipschitz real-valued map $F$ on a separable Fr\'{e}chet space is G\^{a}teaux differentiable at all points except for a Haar null set, i.e., the set of points where $F$ is G\^{a}teaux differentiable is a prevalent set. By Theorem \ref{Thm:GdiffCphi}, $F_\C$ is G\^{a}teaux differentiable at $\phi$ if and only if $\phi\in U_\C(E)$. We conclude that $U_\C(E)$ is prevalent.
\end{proof}

\begin{remark}
    If $E=C(X)$ in \Cref{Thm:PrevalentUniq}, then the conditions for Theorem \ref{Thm:PrevalentUniq} are satisfied, and hence the set $U_\C=\{\phi\in C(X): |M_\C^*(X,T;\phi)|=1\}$ is a prevalent subset of $C(X)$. In combination with Remark \ref{Rmk:Uniqueness1}, we have proved Theorem \ref{Thm:UniquenessIntro}.
\end{remark}

\section{Realization} \label{Sect:Realization}

In this section we take up the question of when a subset of $S\subset M_\C(X,T)$ can be `realized' as the solution set $M_{\C}^*(X,T;\phi)$ for some function $\phi\in C(X)$. In other words, given $S\subset M_\C(X,T)$, can we find a function $\phi\in C(X)$ such that $M^*_\C(X,T;\phi)=S$? By Proposition \ref{Prop:Property M_C^*}, the solution set $M_{\C}^*(X,T;\phi)$ is always a closed face of $M_{\C}(X,T)$. Jenkinson showed in \cite{jenkinson2006every} that every closed face of $M(X,T)$ can be realized as a solution set for the (unconstrained) ergodic optimization problem (\ref{Eqn:EO}).

\begin{theorem}[\cite{jenkinson2006every}, Theorem 4]
\label{Thm:JenRealization}
    Suppose $(X,T)$ is a topological dynamical system and $F$ is a closed face of $M(X,T)$. Then there exists a function $\phi\in C(X)$ such that $M^*(X,T;\phi)=F$. 
\end{theorem}

% The following result is helpful in this context.

% \begin{proposition}[\cite{jenkinson2006every}]
% \label{Prop:JenRepresentation}
%     Suppose $l:M(X,T)\to\R$ is weak$^*$ continuous and affine. Then there exists $g\in C(X)$ such that
%     \[
%     l(\mu)=\int g \, d\mu\quad\text{for all}\ \mu\in M(X,T).
%     \]
% \end{proposition}

%Let $M^*(X,T;\phi)$ be the set of solutions to (\ref{Eqn:EO}), a significant result in \cite{jenkinson2006every} is that $M^*(X,T;\phi)$ is a closed face of $M(X,T)$ for any $\phi\in C(X)$. 
Let $K$ be a convex subset of a topological vector space. In convex analysis, a face $F$ of $K$ is said to be {\it exposed} if there exists an affine functional $l:K\to\R$ such that $l|_F\equiv 0$ and $l|_{K\setminus F}<0$. We are particularly interested in closed faces because any closed face of a compact metrizable simplex is exposed (see \cite{davies1967generalized,edwards1966minimum}). Using this fact in combination with Proposition \ref{Prop:JenRepresentation} above, one may show Jenkinson's result that any closed face of $M(X,T)$ can be realized by a function $\phi\in C(X)$.    

Note that when the facial property (see Definition \ref{Defn:Facial}) holds, $M_\C(X,T)$ is a closed face of $M(X,T)$ and is also a simplex. With this observation, the following result is an easy consequence of \Cref{Thm:JenRealization}.

\begin{proposition}
    Let $(X,T)$ be a topological dynamical system and $\C\subset C(X)$. If $M_\C(X,T)$ is non-empty and the facial property holds, then for any closed face $F$ of $M_\C(X,T)$, there exists a function $\phi\in C(X)$ such that $M_\C^*(X,T;\phi)=F$.
\end{proposition}
\begin{proof}
    When the facial property holds, $M_\C(X,T)$ itself is a closed face of $M(X,T)$, so any closed face $F$ of $M_\C(X,T)$ is also a closed face of $M(X,T)$. Therefore, by \Cref{Thm:JenRealization}, there is a function $\phi_F\in C(X)$ such that $F=M^*(X,T;\phi_F)$. Since $F \subset M_\C(X,T)$, we conclude that $F = M_{\C}^*(X,T;\phi_F)$. 
\end{proof}

%Recall that by \Cref{Prop:finitetype}, if the triple $(X,T,\C)$ has the facial property, then $(X,T,\C)$ has finite type. 
%$ has finite-type constraints whenever the facial property holds. 
% \begin{proposition}
%     If the facial property holds, then there exists $\phi \in C(X)$ such that $M_{\C}(X,T) = M_{\{\phi\}}(X,T)$, and hence $M_{\C}(X,T)$ has finite-type constraints.
% \end{proposition}
% \begin{proof}
% If the facial property holds, then $M_\C(X,T)$ is a closed face of $M(X,T)$ and thus it is exposed. The rest of the proof is exactly the same as the proof of Proposition \ref{Prop:finitetype}. 
% \end{proof}

The hypothesis that the facial property holds may be restrictive. However, if $M_\C(X,T)$ does not satisfy the facial property, then it is not a closed face of $M(X,T)$, and the closed faces of $M_\C(X,T)$ may not be exposed. Next we give a more general realization result for constrained systems $(X,T,\C)$ of finite type. Note that this hypothesis is indeed more general than the facial property by Proposition \ref{Prop:finitetype}. The following notion in convex analysis is useful in this context.

\begin{definition}[Finite codimensional slice]
Suppose $K$ is a convex subset of a vector space and $h_1,\dots,h_n$ are affine functionals on $K$ and that
\[
M_K=\{x\in K: h_i(x)=0,\ i=1,\dots,n\}.
\]
Then $M_K$ is called a finite codimensional slice of $K$. If $K$ is a compact convex set and if $M_K$ is closed in $K$, the we call $M_K$ a closed finite codimensional slice of $K$.
\end{definition}

\begin{lemma}
\label{Prop:finitecodim}
If $M_\C(X,T)$ has finite-type constraints, then it is a finite codimensional slice of $M(X,T)$.    
\end{lemma}
We omit the proof, as it is immediate from the definitions and \Cref{Lemma:McapH}.
%\begin{proof}
    
    % If $M_\C(X,T)$ has finite-type constraints, then there exists a finite set $\mathcal{C}'=\{\phi_1,\dots,\phi_n\}$ such that $M_\C(X,T)=M_{\C'}(X,T)$. Let $J:C(X)\to[C(X)]^{**}$ be the natural embedding and $h_i=J_{\phi_i}|_{M(X,T)}$, and note that $h_i$ is affine on $M(X,T)$. Then
    % \[
    % M_\C(X,T)=\{\mu\in M(X,T): h_i(\mu)=0,\ i=1,\dots,n\}.
    % \]
    % Hence $M_\C(X,T)$ is a closed finite codimensional slice of $M(X,T)$.
%\end{proof}

\begin{remark}
By Proposition \ref{Prop:finitetype}, if $M_\C(X,T)$ is a closed face of $M(X,T)$, then it is a closed slice of codimension one.    
\end{remark}

The following result is helpful in establishing our main realization result below.
\begin{proposition}[\cite{lau1973infinite}, credited to Lazar]
\label{Prop:Lazar}
Suppose $K$ is a compact convex set and that $M$ is a closed finite codimensional slice in $K$. If $F$ is a closed face of $M$, then there exists a closed face $F_1$ of $K$ such that $F=M\cap F_1$. If $F$ is a $G_\delta$ set in $H_1\cap K$, then $F$ is a $G_\delta$ set in $K$.
\end{proposition}

In the proof of the following theorem we show that if $M_\C(X,T)$ has finite-type constraints, then any closed face of $M_\C(X,T)$ is exposed.

\realization*

% \begin{theorem}
%     Let $(X,T)$ be a topological dynamical system and $\C \subset C(X)$. If $(X,T,\C)$ has finite type, then for every closed face $K$ of $M_{\C}(X,T)$, there exists a continuous function $\phi \in C(X)$ such that $K = M_{\C}^*(X,T;\phi)$.  
% \end{theorem}

\begin{proof}
Since $M_\C(X,T)$ has finite type constraints, without loss of generality, we can assume that $\mathcal{C}$ is finite. Then, by Lemma \ref{Prop:finitecodim},  $M_\C(X,T)$ is a finite codimensional slice of $M(X,T)$. Since $M_\C(X,T)$ is closed, by Proposition \ref{Prop:Lazar}, we know that for any closed face $K$ in $M_\C(X,T)$, there exists a closed face $F$ in $M(X,T)$ such that $K=M_\C(X,T)\cap F$. Due to the fact that any closed face in $M(X,T)$ is exposed \cite{davies1967generalized,edwards1966minimum}, $F$ is an exposed face and there is a weak$^*$ continuous affine functional $l$ that is defined on $M(X,T)$ such that $l|_F=0$ and $l_{M(X,T)\setminus F}<0$. By Proposition \ref{Prop:JenRepresentation}, there exists a function $\phi\in C(X)$ such that $l(\mu)=\int\phi \, d\mu$ for all $\mu\in M(X,T)$. Hence the function $\phi \in C(X)$ satisfies $M^*_\mathcal{C}(X,T;\phi)=K$. %Because
%\[
%\sup_{\mu\in M_\C(X,T)}\int(-\phi)\ d\mu
%\begin{cases}
%    =0, & \text{if}\ \mu\in K\\
%    <0, & \text{if}\ \mu\in M_\C(X,T)\setminus K
%\end{cases}.
%\]
\end{proof}

\section{Duality} \label{Sect:Duality}
In this section we consider the dual problem to the problem of ergodic optimization with linear constraints. More specifically, we provide a representation of the dual problem in a form that provides a common generalization of duality results in ergodic optimization \cite[Theorem 2.3]{garibaldi2017ergodic} and Kantorovich duality in optimal transport (see \cite[Chapter 5]{villani2009optimal}). 

The following theorem contains our main duality result. It is closely related to the Kantorovich duality with linear constraints, first established in \cite{zaev2015monge}.

\begin{theorem}
\label{Thm:Kanduality}
Suppose $X$ is a compact metrizable space, $H$ is a linear subspace of $C(X)$,  $W\subset C(X)$ is a linear subspace with $\mathbf{1}\in W$, and $\nu$ is a positive linear functional defined on $W$ with $\nu(\mathbf{1})=1$. Define
\[
\Pi_ H(\nu)=\left\{\mu\in M(X): \mu|_W=\nu, \, \text{and }  \forall f\in   H, \,  \int f \, d\mu=0 \right\}.
\]
Then for any objective function $\phi\in C(X)$, we have
\begin{equation}
\label{Eqn:duality}
\inf_{\mu\in\Pi_ H(\nu)}\int\phi \, d\mu=\sup_{\substack{f+w\leq\phi \\ f\in   H, w\in W}}\nu(w).
\end{equation}
% where $w\in W$. \textcolor{blue}{The statement of this theorem needs to be refined. As it stands, $T$ is defined but then never used again? }
\end{theorem}

Before proving Theorem \ref{Thm:Kanduality}, we first establish the following duality result without linear constraints.
\begin{proposition}
\label{Prop:Kanduality}
Suppose $X$ is a compact metrizable space, $W\subset C(X)$ is a linear subspace with $\mathbf{1}\in W$, and $\nu$ is a positive linear functional defined on $W$ with $\nu(\mathbf{1})=1$. Define
\[
\Pi(\nu)=\{\mu\in M(X): \mu|_W=\nu\}.
\]
Then for any objective function $\phi\in C(X)$, 
\[
\inf_{\mu\in\Pi(\nu)}\int\phi \, d\mu=\sup_{\substack{w\leq\phi,\\ w\in W}}\nu(w).
\]
\end{proposition}
\begin{proof} %Since the linear functional $\nu: W \to \mathbb R$ is bounded, it is continuous with respect to the supremum norm. 
We begin by defining the functional $U: C(X)\to \R$ by setting
\[
U(g)=\inf_{w\in W}\{\nu(w): w\geq g\}.
\]

Let us first show that $U$ is well-defined and bounded. Since $\|g\|\cdot\mathbf{1}\geq g$ and $\|g\|\cdot\mathbf{1}$ is a function in $W$, by the definition of $U$, we have $U(g)\leq \nu(\|g\|\cdot\mathbf{1})=\|g\|\cdot\nu(\mathbf{1})=\|g\|$. Additionally, since $g\geq -\|g\|\cdot\mathbf{1}$ and $-\|g\|\cdot\mathbf{1}\in W$, for any $w \in W$ satisfying $w\geq g\geq -\|g\|\cdot\mathbf{1}$, we have that $w+\|g\|\cdot\mathbf{1}$ is in $W$ and $w+\|g\|\cdot\mathbf{1}\geq 0$. Thus, $\nu(w+\|g\|\cdot\mathbf{1})\geq 0$ by the positivity of $\nu$, which implies $\nu(w)\geq\nu(-\|g\|\cdot\mathbf{1})=-\|g\|$ by linearity. Since $\nu(w)\geq -\|g\|$ for any $w\geq g$, taking the infimum over all $w\in W$ such that $w\geq g$ gives $U(g)\geq -\|g\|$. Therefore, combining the inequalities in this paragraph, we have $|U(g)|\leq \|g\|$ for any $g\in C(X)$, and hence $U$ is well-defined and bounded.

We claim that $U$ is subadditive and positively homogeneous. To prove the subadditivity, let $f,g\in C(X)$, and suppose $w_1, w_2 \in W$ satisfy $w_1 \geq f$ and $w_2 \geq g$. Note that $w_1 + w_2 \in W$ and $w_1 + w_2 \geq f+ g$. Hence we have that
\begin{align*}
U(f+g)&=\inf_{w\in W}\{\nu(w): w\geq (f+g)\}\\
& \leq \nu(w_1+w_2) \\
& = \nu(w_1)+\nu(w_2).
\end{align*}
Taking the infimum over all $w_1$ and $w_2$ satisfying the above conditions, we obtain
\begin{align*}
U(f+g)&=\inf_{w\in W}\{\nu(w): w\geq (f+g)\}\\
&\leq \inf_{w\in W}\{\nu(w): w\geq f\}+\inf_{w\in W}\{\nu(w): w\geq g\}\\
&=U(f)+U(g).
\end{align*}
Next let us show that $U$ is positively homogeneous. Indeed, for any $\alpha\in\mathbb R^+$, we have
\[
\begin{split}
U(\alpha g)&=\inf_{w\in W}\{\nu(w): w\geq\alpha g\}\\
&=\inf_{(\sfrac{w}{\alpha})\in W}\{\alpha \nu(\sfrac{w}{\alpha}): \sfrac{w}{\alpha}\geq g\}\\
&=\alpha U(g).
\end{split}
\]
Additionally, for any $t\in\mathbb R$, we claim that $U(tg)\geq tU(g)$. To verify this claim, it suffices to show that $U(-g)\geq-U(g)$, since then we have $U(tg)=U((-1)(-t)g)\geq-U((-t)g)=tU(g)$ whenever $t<0$. Since $\nu$ is a positive linear functional, we have
\[
\sup_{w\in W}\{\nu(w): w\leq g\}\leq\inf_{w\in W}\{\nu(w): w\geq g\}.
\]
Then since $\inf E=-\sup(-E)$, %the property $U(-g)\geq -U(g)$ is given by
we see that
\[
\begin{split}
U(-g)&=\inf_{w\in W}\{\nu(w): w\geq -g\}=\inf_{w\in W}\{\nu(w): -w\leq g\}\\
&=\inf_{w\in W}\{\nu(-w): w\leq g\}=-\sup_{w\in W}\{\nu(w): w\leq g\}\\
&\geq -\inf_{w\in W}\{\nu(w): w\geq g\}=-U(g).
\end{split}
\]

% Therefore, by combining with positive homogeneity, for any $t\in\mathbb R$ we have
% \[
% U(t\cdot g)
% \begin{cases}
%     =t\cdot U(g) &\text{for}\ t\geq 0,\\
%     \geq -U(-t\cdot g)=t\cdot U(g) &\text{for}\ t<0.
% \end{cases}
% \]

Thus, $U: C(X)\to \mathbb R$ is a positive homogeneous and subadditive functional. Moreover, by the definition of $U$, the linear functional $\nu: W\to\mathbb R$ equals $U$ on the linear subspace $W$. Therefore, by the Hahn-Banach theorem for positive linear functionals \cite[Theorem 2.38]{buhler2018functional}, $\nu$ can be extended to a positive linear functional $P$ on $C(X)$ such that $P\leq U$ on $C(X)$ and $P|_W=\nu$. 

% We claim that $P$ is a positive functional on $C(X)$. If $P$ is not positive, then there exists $g\in  C(X)$ such that $g\geq 0$ and $P(g)<0$. And we have the following result,
% \[
% 0<P(-g)\leq U(-g)=\inf_{w\in W}\{\nu(w): w\geq -g\}\leq 0.
% \]
% It is a contradiction. Therefore, $P$ is positive.

For $\phi\in C(X)$, let us define a new linear operator $\nu_\phi:\{w+t\phi: t\in\mathbb R, w\in W\}\to\mathbb R$ by requiring that $\nu_{\phi}$ be linear and
\begin{equation}
    \begin{cases}
        \nu_\phi|_W=\nu\ \text{and}\ \nu(-\phi)=U(-\phi), &\text{if}\ \phi\notin W\\
        \nu_\phi=\nu, &\text{if}\ \phi\in W
    \end{cases}.
\end{equation}
Notice that $\nu_\phi|_W=\nu$ and $\nu(-\phi)=U(-\phi)$ always hold (regardless of whether $\phi$ is in $W$ or not). By linearity of $\nu_\phi$ and the fact that $U(tg)\geq tU(g)$ for any $t\in\R$, we have
\[
\nu_\phi(t\phi)= (-t) \nu_{\phi}(-\phi) = (-t)U(-\phi)\leq U(t\phi).
\]
We claim that $\nu_\phi$ is bounded by $U$ on its domain. Indeed, for $w\in W$ and $t\in\mathbb R$, we have
\[
\begin{split}
\nu_\phi(w+t\phi)&=\nu_\phi(w)+\nu_\phi(t\phi) \\
& \leq \nu(w)+U(t\phi)\\
&=\nu(w)+\inf_{f\in W}\{\nu(f): f\geq t\phi\}\\
&=\inf_{f\in W}\{\nu(w+f): f\geq t\phi\}\\
&=\inf_{h\in W}\{\nu(h): h\geq w+t\phi\}=U(w+t\phi).
\end{split}
\]
By Hahn-Banach theorem, we can extend $\nu_\phi$ to a linear functional $P_\phi:  C(X)\to \mathbb R$ such that
\[
P_\phi|_{\{w+t\phi: t\in\mathbb R, w\in W\}}=\nu_\phi,\ P_\phi|_W=\nu,\ P_\phi(-\phi)=U(-\phi)\ \text{and}\ P_\phi\leq U.
\]
Since $P$ is bounded by $U$ on $C(X)$, we have
\[
\sup_P P(-\phi)\leq U(-\phi)=\inf_{w\in W}\{\nu(w): w\geq -\phi\},
\]
where the supremum is taken over all possible linear extensions that extend $\nu$ and are bounded by $U$. By linearity of $P_\phi$,
\[
\begin{split}
P_\phi(\phi)&=-P_\phi(-\phi)=-U(-\phi)\\
&=-\inf_{w\in W}\{\nu(w): w\geq -\phi\}=\sup_{w\in W}\{\nu(w): w\leq\phi\}.   
\end{split}
\]
Due to the fact that $\sup_P P(-\phi)\leq\inf_{w\in W}\{\nu(w): w\geq -\phi\}$, we have
\[
-\sup_P P(-\phi)\geq-\inf_{w\in W}\{\nu(w): w\geq -\phi\},
\]
and therefore
\[
\inf_P P(\phi)\geq\sup_{w\in W}\{\nu(w): w\leq \phi\}=P_\phi(\phi).
\]
Since $P_\phi$ extends $\nu$ and is dominated by $U$, we have
\[
\inf_P P(\phi)=P_\phi(\phi)=\sup_{w\in W}\{\nu(w): w\leq\phi\}.
\]

Since $P$ is a positive linear functional on $C(X)$, by the Riesz Representation Theorem for $C(X)^*$ (see \cite[Corollary 14.15]{aliprantis2006infinite}), for $P\in C(X)^*$, there is a unique signed Borel measure $\hat{\mu}$ on $\mathcal{B}(X)$ such that
\[
P(\phi)=\int_X \phi \, d\hat{\mu}\ \text{for all}\ \phi\in  C(X).
\]
For any $w\in W$, since $P$ is an extension of $\nu$, we have $P(w)=\nu(w)=\int_X w \, d\hat{\mu}$, which implies $\hat{\mu}|_W=\nu$.

Moreover, since $\mathbf{1}\in W$ and $P(\mathbf{1})=\nu(\mathbf{1})=1$, we have
\[
P(\mathbf{1})=\int_X\mathbf{1} \, d\hat{\mu}=\hat{\mu}(X)=1,
\]
so $\hat{\mu}$ is a probability measure. Therefore, the positive linear functional $P$ with $P|_W=\nu$ can be identified with a probability measure $\hat{\mu}$ in $\Pi(\nu)$.
% , i.e.,
% \[
% P\simeq\hat{\mu}\in\Pi(\nu).
% \]
Hence, we obtain the following Kantorovich duality:
\[
\inf_{P}P(\phi)=\inf_{\mu\in\Pi(\nu)}\int\phi \, d\mu=\sup_{\substack{w\leq \phi\\w\in W}}\nu(w).
\]
\end{proof}

The next result is a general version of the minmax theorem, which we use in our proof of Theorem \ref{Thm:Kanduality}. A proof of the minmax theorem can be found in \cite{adams2012function}.
\begin{proposition}[\cite{adams2012function}, Theorem 2.4.1]
\label{Prop:minmax}
Let $K$ be a compact convex subset of a Hausdorff topological vector space, $Y$ be a convex subset of an arbitrary vector space, and $h$ be a real-valued function $(\leq +\infty)$ on $K\times Y$, which is lower semi-continuous in $x$ for each fixed $y$, convex on $K$, and concave on $Y$. Then
\[
\min_{x\in K}\sup_{y\in Y}h(x,y)=\sup_{y\in Y}\min_{x\in K}h(x,y).
\]
\end{proposition}

Now we are ready to prove Theorem \ref{Thm:Kanduality}. 
\begin{proof}[Proof of \Cref{Thm:Kanduality}]
% First, we have
% \[
% \begin{split}
% \inf_{\mu\in\Pi_H(\nu)}\int\phi \, d\mu &\geq \inf_{\mu\in\Pi_H(\nu)}\sup_{\substack{f+w\leq\phi\\ w\in W,f\in H}}\int (f+w) \, d\mu\\
% &=\inf_{\mu\in\Pi_H(\nu)}\sup_{\substack{f+w\leq\phi\\ w\in W,f\in H}}\nu(w)=\sup_{\substack{f+w\leq\phi\\ w\in W,f\in H}}\nu(w).
% \end{split}
% \]
% So it suffices to show
% \begin{equation}\label{Eqn:IneqThm6.1}
% \inf_{\mu\in\Pi_H(\nu)}\int\phi \, d\mu\leq \sup_{\substack{f+w\leq\phi\\ w\in W,f\in H}}\nu(w).
% \end{equation}

By Proposition \ref{Prop:Kanduality}, the right hand side of (\ref{Eqn:duality}) can be written as
\begin{equation}\label{Eqn1:Thm6.1}
    \sup_{\substack{f+w\leq\phi \\ w\in W, f\in  H}}\nu(w)=\sup_{f\in  H}\sup_{\substack{w\leq\phi-f\\ w\in W}}\nu(w)=\sup_{f\in  H}\inf_{\mu\in\Pi(\nu)}\int (\phi-f) \, d\mu.
\end{equation}
Let $g : \Pi(\nu) \times H \to \R$ be given by $g(\mu,f)=\int (\phi-f) \, d\mu$. In the following, we show that $g$ satisfies all the assumptions in Proposition \ref{Prop:minmax}.
    First, let us show that for each $f\in H$ fixed, $g(\cdot,f)$ is lower semi-continuous on $\Pi(\nu)$. Indeed, let $\{\mu_k\}$ be a sequence of measures in $\Pi(\nu)$ such that $\mu_k$ converges to $\mu$ in weak$^*$ topology. Then,
    \[
    \lim_{k\to\infty}g(\mu_k,f)=\int(\phi-f) \, d\mu_k=\int(\phi-f) \, d\mu=g(\mu,f).
    \]
    Next, let us show that for each fixed $f\in H$, the function $g(\cdot,f)$ is convex on $\Pi(\nu)$. For any $\mu_1,\mu_2\in\Pi(\nu)$, note that $\alpha\mu_1+(1-\alpha)\mu_2\in\Pi(\nu)$, and
    \[
    \begin{split}
    g(\alpha\mu_1+(1-\alpha)\mu_2,f) &=\int(\phi-f) \, d[\alpha\mu_1+(1-\alpha)\mu_2]\\
    &=\alpha\int(\phi-f) \, d\mu_1+(1-\alpha)\int(\phi-f) \, d\mu_2\\
    &=\alpha g(\mu_1,f)+(1-\alpha)g(\mu_2,f).
    \end{split}
    \]
    Lastly, let us show that for every $\mu\in\Pi(\nu)$, the function $g(\mu,\cdot)$ is concave on $H$. For any $f_1,f_2\in H$ and any $\beta\in(0,1)$,
    \[
    \begin{split}
    g(\mu,\beta f_1+(1-\beta)f_2) &=\int[\phi-(\beta f_1+(1-\beta)f_2)] \, d\mu\\
    &=\beta\int(\phi-f_1) \, d\mu+(1-\beta)\int(\phi-f_2) \, d\mu\\
    &=\beta g(\mu,f_1)+(1-\beta)g(\mu,f_2).
    \end{split}
    \]
Now, let $K=\Pi(\nu)$, $Y=H$ and let $g(\mu,f)=\int (\phi-f) \, d\mu$. By Proposition \ref{Prop:minmax}, we have
\begin{equation}\label{Eqn2:Thm6.1}
    \sup_{f\in  H}\inf_{\mu\in\Pi(\nu)}\int (\phi-f) \, d\mu=\inf_{\mu\in\Pi(\nu)}\sup_{f\in H}\int (\phi-f) \, d\mu.
\end{equation}

We distinguish two cases. First, consider the case when $\Pi_H(\nu)=\varnothing$. Then by convention we have $\inf_{\mu\in\Pi_H(\nu)}\int\phi\, d\mu=\infty$. We claim that $\inf_{\mu\in\Pi(\nu)}\sup_{f\in  H}\int (\phi-f) \, d\mu=\infty$. If $\Pi(\nu)$ is empty, then this equality also holds by convention. Now suppose $\mu \in \Pi(\nu)$. Since $\Pi_H(\nu)$ is empty, $\mu$ is not in $\Pi_H(\nu)$, and therefore there exists $f_1 \in H$ such that $\int f_1 \, d\mu \neq 0$. After possibly replacing $f_1$ by $-f_1$, we assume without loss of generality that $\int f_1 \, d\mu <0$. Letting $f = \alpha f_1$ for arbitrarily large $\alpha >0$, we see that $\sup_{f\in H}\int(\phi-f) \, d\mu=\infty$. Since $\mu \in \Pi(\nu)$ was arbitrary, we obtain $\inf_{\mu\in\Pi(\nu)}\sup_{f\in  H}\int (\phi-f) \, d\mu=\infty$. We therefore see that if $\Pi_H(\nu)=\varnothing$, then
\[
\inf_{\mu\in\Pi(\nu)}\sup_{f\in  H}\int (\phi-f) \, d\mu=\inf_{\mu\in\Pi_H(\nu)}\int\phi\, d\mu,
\]
and combined with (\ref{Eqn1:Thm6.1}) and (\ref{Eqn2:Thm6.1}), we conclude that \Cref{Thm:Kanduality} holds in this case.

Now consider the second case when $\Pi_H(\nu)\neq\varnothing$. Let $\mu \in \Pi_H(\nu)$. Then by definition of $\Pi_H(\nu)$, we have
\begin{equation} \label{Eqn:Flowers}
    \sup_{f\in H}\int(\phi-f)\, d\mu=\int\phi\, d\mu < \infty,
\end{equation}
and hence
\begin{equation} \label{Eqn:Butterflies}
    \inf_{\mu\in\Pi(\nu)}\sup_{f\in H}\int (\phi-f) \, d\mu < \infty.
\end{equation}
Now let $\mu \in \Pi(\nu) \setminus \Pi_H(\nu)$. 
% Then there exists $f_1\in H$ such that $\int f_1 \, d\mu<0$ and we can choose $f=\alpha f_1$, when $\alpha\to\infty$, and hence 
By the similar argument to the case $\Pi_H(\nu)=\varnothing$, we have $\sup_{f\in H}\int(\phi-f) \, d\mu=\infty$. By this observation and (\ref{Eqn:Butterflies}), we see that the infimum in (\ref{Eqn:Butterflies}) can be taken over the smaller set $\Pi_H(\nu)$ (instead of $\Pi(\nu)$), i.e.,
\begin{equation*}
    \inf_{\mu\in\Pi(\nu)}\sup_{f\in H}\int (\phi-f) \, d\mu = \inf_{\mu\in\Pi_H(\nu)}\sup_{f\in H}\int (\phi-f) \, d\mu.
\end{equation*}
Combining this equation with (\ref{Eqn:Flowers}), we obtain
\begin{equation}\label{Eqn3:Thm6.1}
    \inf_{\mu\in\Pi(\nu)}\sup_{f\in H}\int (\phi-f) \, d\mu = \inf_{\mu\in\Pi_H(\nu)}\sup_{f\in H}\int (\phi-f) \, d\mu = \inf_{\mu \in \Pi_H(\nu)} \int \phi \, d\mu.
\end{equation}

% Let $\mu\in\Pi(\nu)$ be fixed in the right hand side above. First consider the case when $\mu\in\Pi_H(\nu)$. Then we have $\sup_{f\in H}\int(\phi-f)\, d\mu=\int\phi\, d\mu$.  
% % and then taking the infimum over $\Pi(\nu)$ is either finite when $\Pi(\nu)$ is non-empty or $+\infty$ otherwise. 
% Now consider the case hen $\mu\notin\Pi_H(\nu)$. Then there exists $f_1\in H$ such that $\int f_1 \, d\mu<0$ and we can choose $f=\alpha f_1$, when $\alpha\to\infty$, and hence $\sup_{f\in H}\int(\phi-f) \, d\mu=\infty$. Taking infimum over $\Pi(\nu)$ always gives $+\infty$.
% \[
% \sup_{f\in H}\int(\phi-f)\, d\mu=
% \begin{cases}
%     \int\phi\, d\mu, &\text{if}\ \mu\in\Pi_H(\nu)\\
%     +\infty, &\text{otherwise}
% \end{cases},
% \]
% since if $\mu\notin\Pi_H(\nu)$, there exists $f_1\in H$ such that $\int f_1 \, d\mu<0$ and we can choose $f=\alpha f_1$, when $\alpha\to\infty$, $\sup_{f\in H}\int(\phi-f) \, d\mu\to\infty$. 
Finally, by (\ref{Eqn1:Thm6.1}), (\ref{Eqn2:Thm6.1}) and (\ref{Eqn3:Thm6.1}) we conclude \Cref{Thm:Kanduality}.
\end{proof}

Theorem \ref{Thm:Kanduality} can be applied to dynamical systems and optimal transport by making appropriate choices of $H$ and $W$. Below we present some examples.

\begin{example}[The duality of (\ref{Eqn:EO})]
\label{Example:dualEO}
    Suppose $(X,T)$ is a topological dynamical system. If we let $  H=\{g\circ T-g: g\in C(X)\}$, $W=\{c \cdot \mathbf{1}: c\in\mathbb R\}$,  $\nu=\mathrm{Id}$ (the functional mapping the constant function $\lambda \cdot \mathbf{1}$ to $\lambda$), then
    \[
    \Pi_ H(\nu)=\{\mu\in M(X):\mu\circ T^{-1}=\mu\ \text{and}\ \mu(c \cdot \mathbf{1})=\nu(c \cdot \mathbf{1}) = c\}=M(X,T).
    \]
    By Theorem \ref{Thm:Kanduality}, we have
    \begin{equation}
    \label{Eqn:DualityofEO}
        \sup_{\mu\in M(X,T)}\int\phi \, d\mu=\inf_{\substack{g-g\circ T+c\geq \phi\\g\in C(X),c\in\R}}c.
    \end{equation}
    
\end{example}

\begin{remark}
    The duality result in \cite[Theorem 2.3]{garibaldi2017ergodic} is
    \begin{equation}
    \label{Eqn:DualityGarbridini}
        \sup_{\mu\in M(X,T)}\int\phi \, d\mu=\inf_{g\in C(X)}\sup_{x\in X}[g\circ T(x)-g(x)+\phi(x)].
    \end{equation}
    Notice that the right hand side of (\ref{Eqn:DualityGarbridini}) is equal to the right hand side of (\ref{Eqn:DualityofEO}) in \Cref{Example:dualEO}. Therefore, we have recovered \cite[Theorem 2.3]{garibaldi2017ergodic}.
\end{remark}

\begin{example}[The duality of (\ref{Eqn:EOWC})]
\label{Example:dualEOWC}
    Suppose $(X,T)$ is a topological dynamical system and $\C\subset C(X)$. Let $W=\{g\circ T-g+c: g\in C(X), c\in\R\}$. Let $H$ be the smallest closed linear subspace of $C(X)$ containing $\C$, and let $\nu\in M(X,T)$ be any $T$-invariant probability measure. Then
    \[
    \Pi_H(\nu)=\{\mu\in M(X): \mu\circ T^{-1}=\mu\ \text{and}\ \mu(f)=0\ \text{for any}\ f\in\C\}=M_\C(X,T).
    \]
    By Theorem \ref{Thm:Kanduality}, we have
    \[
    \sup_{\mu\in M_\C(X,T)}\int\phi \, d\mu=\inf_{\substack{f+g-g\circ T+c\geq \phi\\f\in H,g\in C(X),c\in\R}}c.
    \]
    This establishes the \Cref{Thm:Duality}.
\end{example}

\begin{example}[Duality of ergodic optimal transport]
\label{Example:LinearOT}
    Let $(X_i,T_i),\ i=1,2,\dots,n$ be a sequence of topological dynamical systems, where $X_1,\dots,X_n$ are compact and metrizable spaces and $T_i:X_i\to X_i$ are continuous transformations for $i=1,\dots,n$. Let $X=X_1\times\dots\times X_n$,  $T=T_1\times\dots\times T_n$ and $\pi_i: X\to X_i$ be the projection map. If we let 
    \[
    H=\{\varphi\circ T-\varphi: \varphi\in C(X)\},
    \]
    \[
    W=\{f_1\circ\pi_1+\dots+f_n\circ\pi_n: f_1\in C(X_1),\dots,f_n\in C(X_n)\},
    \]
    and $\nu=\mu_1\oplus\mu_2\oplus\dots\oplus\mu_n$. Then
    \[
    \Pi_ H(\nu)=\{\mu\in M(X): \mu\circ\pi_i^{-1}=\mu_i,\ \mu\circ T^{-1}=\mu\}=M(X,T)\cap\Pi(\mu_1,\dots,\mu_n)
    \]
    is the set of invariant couplings. By Theorem \ref{Thm:Kanduality}, for any cost function $c\in C(X)$,
    \[
    \inf_{\pi\in M(X,T)\cap\Pi(\mu_1,\dots,\mu_n)}\int c \, d\pi=\sup_{\substack{\varphi\circ T-\varphi+f_1\circ\pi_1+\dots+f_n\circ\pi_n\leq c\\f_1\in C(X_1),\dots,f_n\in C(X_n)\\ \varphi\in C(X)}}\sum_{i=1}^n\int f_i \, d\mu_i. 
    \]
\end{example}

\begin{remark}
Notice that Example \ref{Example:LinearOT} recovers \cite[Theorem 2.1]{zaev2015monge} in the setting of compact metric spaces. We note that applies in the more general setting of Polish spaces. 
%Therefore, Theorem \ref{Thm:Kanduality}, with a more rigid assumption ($X$ is a compact metric space) than the classical optimal transport theory ($X$ is a Polish space), implies \cite[Thm. 2.1]{zaev2015monge}. 
Also, if we let $H=\{0\}$ be the trivial linear subspace, then the duality in \Cref{Example:LinearOT} becomes the regular Kantorovich duality.
\end{remark}

\begin{example}[Duality of relative ergodic optimization]
Suppose $(X,T)$ and $(Y,S)$ are topological dynamical systems. Let $\pi: X\to Y$ be a factor map and $\nu\in M(Y,S)$. Let $H=\{g\circ T-g: g\in C(X)\}$ and $W=\{f\circ\pi: f\in C(Y)\}$, then
\[
\Pi_ H(\nu)=\{\mu\in M(X,T):  \mu\circ\pi^{-1}=\nu\}:=M_\nu(X,T).
\]
By Theorem \ref{Thm:Kanduality}, we have the following duality result:
\[
\sup_{\mu\in M_\nu(X,T)}\int \phi \, d\mu=\inf_{\substack{g-g\circ T+f\circ\pi\leq \phi\\ f\in C(Y),g\in C(X)}} \int f \, d\nu.
\]
\end{example}

\noindent
\textbf{Acknowledgments.} KM gratefully acknowledges support from the National Science Foundation grants DMS-2436227 and DMS-2413929.

\bibliographystyle{amsplain}
\bibliography{references}

\end{document}